\documentclass[preprint]{elsarticle}

\usepackage{natbib}
\usepackage{geometry}
\usepackage{fleqn}
\usepackage{graphicx}
\usepackage{hyperref}
\usepackage{amsmath, amsthm, amssymb, amsfonts,mathrsfs}
\usepackage{braket}
\usepackage{color}
\usepackage{bbold}

\title{Gaussian quantum Markov semigroups on finitely many modes admitting a normal invariant state}
\date{}

\author[1]{Federico Girotti}
\ead[1]{federico.girotti@polimi.it}
\affiliation[1]{organization={Dipartimento di Matematica, Politecnico di Milano},
addressline={Via Edoardo Bonardi 9},
postcode={20133},
city={Milano},
country={Italy}}
\cortext[1]{federico.girotti@polimi.it}

\author[2]{Damiano Poletti}
\ead[2]{damiano.poletti@unige.it}
\affiliation[2]{organization={Dipartimento di Matematica, Universita di Genova},
addressline={Via Dodecaneso 35},
postcode={16146},
city={Genova},
country={Italy}}
\cortext[2]{damiano.poletti@unige.it}

\newtheorem{defi}{Definition}
\newtheorem{lemma}[defi]{Lemma}
\newtheorem{coro}[defi]{Corollary}
\newtheorem{prop}[defi]{Proposition}
\newtheorem{theo}[defi]{Theorem}
\newtheorem{rem}{Remark}

\theoremstyle{definition}

\def\hh{\mathfrak{h} }
\def\nn{\mathbb{N}}
\def\CC{\mathbb{C}}
\def\RR{\mathbb{R}}
\def\FS{\Gamma(\CC^d)}
\def\TT{\mathcal{T}}
\def\NN{\mathcal{N}(\TT)}
\def\hh{\mathfrak{h}}
\def\tr{{\rm tr}}
\def\LD{L^2_s(\omega_\infty)}

\begin{document}
\begin{abstract}
Gaussian quantum Markov semigroups (GQMSs) are of fundamental importance in modelling the evolution of several quantum systems. Moreover, they represent the noncommutative generalization of classical Orsntein-Uhlenbeck semigroups; analogously to the classical case, GQMSs are uniquely determined by a ``drift" matrix $\mathbf{Z}$ and a ``diffusion" matrix $\mathbf{C}$, together with a displacement vector $\mathbf{\zeta}$. In this work, we completely characterize those GQMSs that admit a normal invariant state and we provide a description of the set of normal invariant states; as a side result, we are able to characterize quadratic Hamiltonians admitting a ground state. Moreover, we study the behavior of such semigroups for long times: firstly, we clarify the relationship between the decoherence-free subalgebra and the spectrum of $\mathbf{Z}$. Then, we prove that environment-induced decoherence takes place and that the dynamics approaches an Hamiltonian closed evolution for long times; we are also able to determine the speed at which this happens. Finally, we study convergence of ergodic means and recurrence and transience of the semigroup.  
\end{abstract}
\maketitle

\makeatletter
\def\ps@pprintTitle{%
  \let\@oddhead\@empty
  \let\@evenhead\@empty
  \let\@oddfoot\@empty
  \let\@evenfoot\@oddfoot
}
\makeatother

\section{Introduction}

Gaussian quantum Markov semigroups (GQMSs) form a class of evolutions on bounded linear operators on the bosonic Fock space $\FS$ (in this work we will ony deal with finitely many modes) that play a central role in the study of quantum mechanical systems, since they provide convenient models for the descriptions of quantum optical experiments, atomic ensembles, quantum memories and many others; therefore, they have been intensively studied in the field of quantum controlled systems under the name of linear quantum systems (see \cite{DZ22} and references therein). From a mathematical point of view, they represent the noncommutative counterpart of classical Ornstein-Uhlenbeck semigroups (see \cite{CFL00,LMP20,FPSU24}).

The name Gaussian comes from the fact that GQMSs can be characterized as the class of semigroups that preserve the set of quantum Gaussian states (\cite{PO22}): given an initial quantum Gaussian state $\rho$ with mean $\mathbf{m} \in \RR^{2d}$ and covariance $\mathbf{\Sigma} \in M_{2d}(\RR)$, its evolution under the semigroup is a family of quantum Gaussian states $(\rho_t)_{t \geq 0}$ with parameters $m_t$ and $\mathbf{\Sigma}_t$ satisfying the following Cauchy problem:
\begin{align}\label{eq:par}
&\frac{d\mathbf{m}_t}{dt}=\mathbf{Z}^*\mathbf{m}_t-\mathbf{\zeta}, \quad \mathbf{m_0}=\mathbf{m},\\
&\frac{d\mathbf{\Sigma}_t}{dt}=\mathbf{Z}^*\mathbf{\Sigma}_t + \mathbf{\Sigma}_t \mathbf{Z} + \mathbf{C}, \quad \mathbf{\Sigma}_0=\mathbf{\Sigma}\nonumber\end{align}
for some $\mathbf{Z}, \mathbf{C} \in M_{2d}(\RR)$ and $\mathbf{\zeta} \in \RR^{2d}$ characterizing the semigroup.

Moreover, their generator can be formally written in a GKSL form with a quadratic Hamiltonian and linear jump operators, where quadratic and linear are to be understood in terms of creation and annihilation operators.

\bigskip The first step in the study of dynamical systems consists in determining whether they admit any equilibria and characterizing their structure. This was the initial motivation of this article: determine when a GQMS admits a normal invariant state and what is the structure of normal invariant states. One of the main contributions of this work is Theorem \ref{thm:main}, where we provide a complete characterization of those GQMSs admitting a normal invariant state in terms of $\mathbf{Z}$, $\mathbf{C}$ and $\mathbf{\zeta}$ or, equivalently, in terms of the Hamiltonian and jump operators. Our investigation points out that there are two building blocks: the quantum harmonic oscillator, i.e. the Hamiltonian evolution driven by a multiple of the number operator, and the case when $\mathbf{Z}$ is stable (let us call it a \textit{stable GQMS}), i.e. the spectrum of $\mathbf{Z}$ is contained in the half-plane of complex numbers with strictly negative real part. Every GQMS admitting a normal invariant state can be decomposed into the tensor product of several quantum harmonic oscillators with possibly different frequencies and a stable GQMS. Since GQMS are the open quantum systems counterpart of quadratic hamiltonian dynamics, a side result of our investigation is the characterization of quadratic Hamiltonians admitting a ground state (Corollary \ref{coro:quadhami}).

\bigskip Normal invariant states of the quantum harmonic oscillator are convex combinations of number states; in the case of the tensor product of several quantum harmonic oscillators, this continues to hold unless the frequencies are not $\mathbb{N}$ linear independent: in this case, the only difference is that one needs to include pure states supported on some particular linear combinations of number states and their closed convex hull. On the other hand, in the case of stable $\mathbf{Z}$, there exists a unique quantum Gaussian invariant state whose mean and covariance are given by the only fixed points of Eq. \eqref{eq:par}. Theorem \ref{theo:invst} completely characterizes the set of normal invariant states of GQMS showing that they are unitarily equivalent to the tensor product of the unique quantum Gaussian invariant state of the stable part and the normal invariant states of the quantum harmonic oscillators. 

\bigskip Interestingly, the family of GQMS admitting a normal invariant state turns out to be well behaved in terms of features of the dynamics for long times. First of all, the properties of the semigroup of being irreducible or possessing a faithful normal invariant state admit easy characterizations in terms of $\mathbf{Z}$, $\mathbf{C}$ and $\mathbf{\zeta}$ (Corollaries \ref{coro:faithful} and \ref{coro:irr}).

Moreover, the decoherence-free subalgebra $\NN$, which is the biggest $W^*$-algebra on which the semigroup acts as a group of $^*$-automorphisms (Proposition \ref{prop:auto}), can be characterized in terms of the eigenspaces of $\mathbf{Z}$ corresponding to purely imaginary eigenvalues and can be shown to be a factor of type $I$ (Theorem \ref{prop:perif} and Corollary \ref{coro:Ifactor}). The decoherence-free subalgebra is a central object in the study of long-time dynamics of quantum Markov semigroups and its connection to Theorem \ref{thm:main} is the following: it turns out that the whole algebra of bounded operators $B(\FS)$ factorizes as $B(\FS)=\NN \overline{\otimes} {\cal A}$, where ${\cal A}$ is another $W^*$-algebra which is invariant for the semigroup, moreover the action of the semigroup restricted to $\NN$ is unitarily equivalent to the tensor product of several quantum harmonic oscillators, while the dynamics restricted to ${\cal A}$ is unitarily equivalent to a stable GQMS. If we denote by $\TT:=(\TT_t)_{t \geq 0}$ the GQMS we are studying, we can prove (Theorem \ref{thm:eid}) that there exists a normal conditional expectation ${\cal E}:B(\FS) \rightarrow \NN \otimes \mathbb{1}$ such that for every $x \in B(\FS)$
\begin{equation} \label{eq:introdeco}
{\rm w}^*\text{-}\lim_{t \rightarrow +\infty} \TT_t(x)-\TT_t{\cal E}(x)=0.
\end{equation}
This shows that for very initial operator $x$, the dynamics asymptotically approaches the Hamiltonian evolution of ${\cal E}(x)$ driven by several harmonic oscillators. This implies that environment-induced decoherence in the sense of \cite{BO03} takes place and the dynamics asymptotically reduces to a closed dynamics on a subsystem of the original space. Making use of an equivalence between a generalized Poincar\'e inequality and the exponential decay of the semigroup restricted to a suitable subspace (\ref{app:PI}), we can show that the speed at which the convergence in Eq. \ref{eq:introdeco} takes place is completely governed by the spectral gap of the ergodic restriction to ${\cal A}$; we recall that the spectral gap of stable GQMS was explicitly found in \cite{FPSU24}.

\bigskip Finally, we are able to describe the convergence of ergodic means and determine the decomposition of the system space into its positive recurrent and transient parts (the null recurrent one turns out to be trivial).

\bigskip We recall that the study of normal invariant states for GQMSs and the convergence of any initial state to equilibrium has already been considered in the literature, even for more general semigroups, namely quasi-free semigroups (\cite{EL77,DVV79,He10,BW24}): see for instance \cite{DVV79,TN22, FP24,BW24}. However, only the stable case was considered so far.

Natural directions for future investigations are trying to conduct the same analysis we did in this work in the case of more general quasi-free semigroups (\cite{BW24}) and in the case of infinitely many modes.

\bigskip The structure of the paper is the following: in Section \ref{sec:np} we fix the notation and we recall some known results and definitions. Section \ref{sec:df} focuses on the study of the decoherence-free subalgebra. In Section \ref{sec:char} we prove the characterization of GQMSs with a normal invariant state. Section \ref{sec:stris} contains the description of the set of normal invariant states and the alternative characterizations of irreducibility and the property of the semigroup of possessing a faithful normal invariant state. In Section \ref{sec:EID} and \ref{sec:ds} the long time behavior of the semigroup is described and we show that the decoherence speed is the same as the one for the completely dissipative part. Section \ref{sec:et} deals with the convergence of ergodic means and the study of recurrence and transience. Finally, in Section \ref{sec:cOU}, we compare some of the results we obtained with their counterparts for classical Ornstein-Uhlenbeck semigroups.

\section{Notation and preliminaries} \label{sec:np}

In this section we will set the notation and recall the main definitions and results that are needed in order to read this work.

\textbf{Bosonic Fock space.} Let $\mathfrak{h}:=\FS$ be the symmetric or bosonic Fock space over $\CC^d$; we recall that this is the closed subspace of the free Fock space
$$
\bigoplus_{n\geq 0} (\CC^d)^{\otimes n}
$$
generated by exponential vectors, i.e. those vectors of the form
$$
e_z=\sum_{n \geq 0} \frac{z^{\otimes n}}{\sqrt{n!}}, \quad z \in \CC^d.
$$
In the physical literature normalized exponential vectors $e^{-\|z\|^2/2}e_z$ are usually called coherent vectors and $\ket{e_0}\bra{e_0}$ is known as the vacuum state. We recall that $\FS$ is isometrically isomorphic to $d$ copies of $\Gamma(\CC)$ via the following correspondence:
\begin{align*}
    \FS &\rightarrow \Gamma(\CC) \otimes \cdots \otimes \Gamma(\CC)\\
    e_z &\mapsto e_{z_1} \otimes \cdots \otimes e_{z_d}
\end{align*}
where $z_i$'s are the coordinates of $z \in \CC^d$ in any orthonormal basis. Every $\Gamma(\CC)$ correspond to a mode, is isomorphic to $\ell^2(\nn)$ and has its own creation, annihilation and number operators; let $\{e(n_1,\dots,n_d):=e(n_1) \otimes \cdots e(n_d)\}_{n_1,\dots, n_d \in \nn}$ be the canonical orthonormal basis for $\ell^2(\mathbb{N})^{\otimes d}\simeq \FS$, then the annihilation, creation and number operators corresponding to the $j$-th mode $a_j$, $a^\dagger_j$, $N_j$ act in the following way on the basis elements:
\begin{align*}
&a_je(n_1,\dots,n_d)=\sqrt{n_j}e(n_1,\dots,n_{j-1}, n_j-1, n_{j+1}, \dots , n_d),\\
&a^\dagger_je(n_1,\dots,n_d)=\sqrt{n_j+1}e(n_1,\dots,n_{j-1}, n_j+1, n_{j+1}, \dots,  n_d),\\
&N_je(n_1,\dots,n_d)=n_j e(n_1, \dots,  n_d).
\end{align*}
We recall that $N_j=a_j^\dagger a_j$ and that the linear space $D$ of finite linear combinations of the vectors of the canonical basis is an essential domain for all such operators.

However, we will rarely work with unbounded operators, dealing instead with Weyl operators: given $z \in \CC^d$, the corresponding Weyl operator is the unique unitary operator acting in the following way on coherent vectors:
\begin{align} \label{eq:Wexpo}
W(z):\FS &\longrightarrow \FS\\
e^{-\frac{\|w\|^2}{2}}e_w &\longmapsto e^{-\frac{\|w+z\|^2}{2}}e_{w+z}.
\end{align}
Weyl operators satisfy the exponential form of canonical commutation relations, i.e.
\begin{equation} \label{eq.expccr}
W(z+w)=e^{i\Im(\langle z,w \rangle)}W(z)W(w), \quad z,w \in \CC^d.
\end{equation}
Moreover, the set of Weyl operators is ${\rm w}^*$-dense in $B(\hh)$. To any $z \in \CC^d$, using Stone's theorem we can associate the unique self adjoint operator $R(z)$ which is the generator of the strongly continuous group $t \mapsto W(tz)$; in this sense, we will write $W(z)=e^{iR(z)}$. $R(z)$'s are called quadratures or field operators; let $\{f_1,\dots,f_d\}$ be the canonical basis of $\CC^d$, then one can check that
$$
R(f_j)=-p_j, \quad R(if_j)=q_j, \quad j=1,\dots, d,$$
where $q_i$ and $p_i$ are position and momentum observables, respectively, corresponding to the $j$-th mode. In general, one can see that
\begin{equation} \label{eq:fields}\sum_{j=1}^{d}\Re(z_j) q_j-\Im(z_j) p_j\subseteq R(z),
\end{equation}
where the sum is defined on a common domain for the operators involved.

\bigskip \textbf{Symplectic structure of $\CC^d$.} Eq. \eqref{eq:fields} hints that, when one deals with field operators, it might be convenient to consider the real Hilbert space structure of $\CC^d$ as well and the following identification with $\RR^{2d}$:
\begin{align*}
\CC^d &\longrightarrow \RR^{2d}\\
x+iy &\longmapsto \begin{pmatrix} x \\ y\end{pmatrix}.
\end{align*}
The real inner product on $\RR^{2d}$ corresponds to
$$\langle z,w \rangle_\RR:=\Re(\langle z,w \rangle), \quad z,w \in \CC^d.$$
We will use the bold font $\mathbf{z}$ to denote the vector in $\RR^{2d}$ corresponding to $z \in \CC^d$. Given a real linear vector subspace $V$ in $\CC^{d}$ we will denote by ${\cal V}$ its image via the identification above, i.e.
$$
{\cal V}:=\left \{ \mathbf{z} \in \RR^{2d} : \, z \in V\right \}.
$$
Let $A$ be a real linear operator on $\CC^d$, we will denote by $A^\sharp$ its adjoint with respect to $\langle \cdot, \cdot \rangle_{\RR}$; moreover, one can always write $A$ in the form
$$
Az=A_1z+A_2\overline{z},$$
where $A_1$, $A_2$ are complex linear operators on $\CC^d$. As an operator on $\RR^{2d}$, $A$ reads as
$$\mathbf{A}=\begin{pmatrix} \Re(A_1)+\Re(A_2) & \Im(A_2)-\Im(A_1) \\
\Im(A_1)+\Im(A_2) & \Re(A_1)-\Re(A_2)
\end{pmatrix}.$$
With an abuse of notation we will denote by $\mathbf{A}$ its complexification as well, acting on $\CC^{2d}$ and we will call spectrum of $\mathbf{A}$, denoted by ${\rm Sp}(\mathbf{A})$, the set of those $z \in \CC$ such that $z-\mathbf{A}$ is not invertible as an operator on $\CC^{2d}$.

The imaginary part of the complex inner product on $\CC^d$, which appears in the commutation relations between Weyl operators (Eq. \eqref{eq.expccr}), is a nondegenerate symplectic form on $\CC^d$, i.e. it is a bilinear form satisfying the following two requirements:
\begin{itemize}
\item $\Im(\langle z,w\rangle=-\Im(\langle w,z\rangle)$, $z,w \in \CC^d$;
\item $\Im(\langle z,w\rangle )=0$ for all $w \in \CC^d$ if and only if $z=0$.
\end{itemize}
We will use the notation $\sigma(z,w):=\Im(\langle z,w \rangle)$. One can immediately check that the symplectic form in $\RR^{2d}$ reads as
$$\sigma(z,w)=\langle \mathbf{z}, \mathbf{J} \mathbf{w}\rangle, \text{ where } \mathbf{J}=\begin{pmatrix} O_{d} & I_{d} \\
-I_d & O_d \end{pmatrix}.$$
$O_d$ (resp. $I_d$) is the $d\times d$-matrix with all entries equal to $0$ (resp. the $d$-dimensional identity matrix). $\mathbf{J}$ is called symplectic matrix. Two vectors $z,w \in \CC^d$ are said to be symplectically orthogonal if $\sigma(z,w)=0$; consequently, two real subspaces $V,W \subseteq \CC^d$ are said to be symplectically orthogonal if every $z \in V$ is symplectically orthogonal to every $w \in W$.

\bigskip \textbf{Symplectic transformations and metaplectic representation.} A symplectic transformation is a real linear transformation $A$ that preserve the symplectic structure, i.e.
$$\sigma( A z, A w )=\sigma ( z, w  ), \quad \forall z, w \in \CC^d.$$

When they are looked at as acting on $\RR^{2d}$, symplectic transformations are characterized by the following simple condition:
\begin{equation} \label{eq:symp1}
\mathbf{A}^*\mathbf{J}\mathbf{A}=\mathbf{J}.
\end{equation}
The set of symplectic transformations is a real Lie group, whose Lie algebra (those real linear operators $\mathbf{Z}$ such that $e^{t\mathbf{Z}}$ is a symplectic transformation for every $t \in \RR$) is composed by those matrices such that
\begin{equation} \label{eq:symp2}
\mathbf{Z}^* \mathbf{J}+\mathbf{J} \mathbf{Z}=\mathbf{0}.
\end{equation}
Eq. \eqref{eq:symp2} is obtained differentiating Eq. \eqref{eq:symp1}. Any symplectic transformation induces a $^*$-automorphism on $B(\hh)$ which acts in the following way on Weyl operators:
$$W(z)\mapsto W(Az), \quad \forall z \in \CC^d.
$$
Using the uniqueness of irreducible representations of the algebra generated by Weyl operators (it is the Stone-Von Neumann Theorem \cite[Theorem 5.3.1.]{Ho11}), one has that there exists a unitary transformation $U(A):\FS \mapsto \FS$ such that
$$W(Az)=U(A)W(z)U(A)^*, \quad \forall z \in \CC^d.$$
$U(A)$ is uniquely defined except for a complex phases; it is not possible to fix a phases such that $U(A)U(B)=U(AB)$ for every symplectic transformations $A$ and $B$, but one can fix the phases in a way that ensures the weaker requirement $U(A)U(B)=\pm U(AB)$. $U(A)$ is called the metaplectic representation of the symplectic group and we refer to Chapter 4, Section 2 in \cite{Fo89} for a detailed discussion.

Any splitting of $\CC^d$ into symplectically orthogonal subspaces
$$\CC^d=V_1 \oplus V_2$$
induces a factorization at the level of the Fock space. Indeed, one can show that there always exists a symplectic transformation $M$ such that
$$\langle Mz,M w \rangle =0, \quad \text{ for all } z\in V_1, \, w \in V_2.$$
Let us define $\tilde{V}_i=M(V_i)$ for $i=1,2$. One can easily check that the following mapping is unitary:
\begin{align*}
    \tilde{U}:\Gamma(\tilde{V}_1) \otimes \Gamma(\tilde{V}_2)&\rightarrow \FS\\
    e(z_1) \otimes e(z_2) &\mapsto U(M)^* e(z_1+z_2).
\end{align*}
$\tilde{U}$ provides us with a convenient representation for the $W^*$-algebras generated by $\{W(z) : \, z \in V_i\}$ for $i=1,2$, since
$$\tilde{U}^{-1} W(z_1) \tilde{U}=W(Mz_1) \otimes \mathbf{1}, \quad \tilde{U}^{-1} W(z_2) \tilde{U}=W(Mz_2) \otimes \mathbf{1}, \quad z_i \in V_i, \, i=1,2.$$

\bigskip \textbf{Gaussian states.} We will use the notation $L^1(\hh)$ for the space of trace class operators. It is well known that every normal state $\varphi$ on $B(\hh)$ is represented by a unique trace class operator $\rho \in L^1(\hh)$ in the following way:
$$
\varphi(x)=\tr(\rho x), \quad x \in B(\hh).$$
$\rho$ is called density operator associated to $\varphi$ and it is positive semidefinite and with trace $1$. With an abuse of notation we will often identify the state with its density operator. Any normal state on $B(\hh)$ is uniquely determined by the analogous of the characteristic function in this setting (\cite[Theorem 5.3.3]{Ho11}), which is given by
$$
\hat{\rho}(z)=\tr(W(z)\rho), \quad z \in \CC^d.$$ In analogy to the classical case, one calls a state $\rho$ a quantum Gaussian state if the characteristic function has the following form:

$$\hat{\rho}(z)=e^{-i\langle \zeta,z \rangle_\RR -\frac{1}{2}\langle z, S z \rangle_\RR}, \quad z \in \CC^d,$$
for some $\zeta \in \CC^d$ and for some positive real linear operator $S$ acting on $\CC^d$; $\zeta$ is said to be the mean of $\rho$ and $S$ is called covariance matrix. Notice that $\hat{\rho}(z)$ is the the characteristic function of the field operator $R(z)$, therefore $\rho$ is a Gaussian state if and only if all the field operators have a Gaussian law in the state $\rho$. An example of Gaussian state is the vacuum state, since one has (using Eq. \eqref{eq:Wexpo})
\begin{equation}
    \langle e_0, W(z) e_0 \rangle =e^{-\frac{\|z\|^2}{2}}.
\end{equation}

While there are no restrictions on the mean, a necessary and sufficient condition for $S$ in order to be the covariance matrix of a quantum Gaussian state is the following requirement:
$$
\mathbf{S}+i\mathbf{J} \geq 0$$
as complex linear operators on $\CC^{2d}$. This is equivalent to Heisenberg uncertainty principle for field operators.

\bigskip \bigskip \textbf{Gaussian quantum Markov semigroups.} A Quantum Markov semigroup $\TT:=\{\TT_t\}_{t \geq 0}$ is a ${\rm w}^*$-continuous semigroup of completely positive, identity preserving, ${\rm w}^*$-continuous maps on $B(\hh)$. The predual semigroup $\TT_* = \{\TT_{*t}\}_{t\geq0}$ acts on trace class operators and is a strongly continuous contraction semigroup. A quantum Markov semigroup is called Gaussian if it maps Gaussian states into Gaussian states; the class of Gaussian quantum Markov semigroups (GQMSs) can be completely characterized either through their explicit action on Weyl operators or through their generator. Let $L_\ell$, $H$ be the operators on $\hh$ defined on the domain $D$ by the following expressions:
\begin{align}
&H=\sum_{k,j=1}^{d} \left ( \Omega_{jk} a^\dagger_j a_k + \frac{\kappa_{jk}}{2}a_j^\dagger a_k^\dagger+\frac{\overline{\kappa}_{jk}}{2}a_j a_k\right )+\frac{1}{2}\sum_{j=1}^{d} \zeta_j a_j^\dagger + \overline{\zeta}_j a_j,\\ \label{eq:hami}
&L_\ell=\sum_{j=1}^{d} \overline{v}_{\ell j}a_j+u_{\ell j}a_j^\dagger, \quad l=1,\dots, m
\end{align}
where $\Omega \in M_d(\CC)$ is hermitian, $\kappa \in M_d(\CC)$ is symmetric, $\zeta \in \CC^d$, $m \leq 2d$ and $U,V \in M_{m\times d}(\CC)$.

For all $x \in B(\hh)$ consider the following quadratic form with domain $D\times D$
\begin{equation} \label{eq:fg} \begin{split}\mathfrak{L}(x)[\xi^\prime, \xi] &= i \langle H\xi^\prime, x\xi\rangle -i \langle \xi^\prime, xH\xi \rangle\\
&-\frac{1}{2}\sum_{\ell=1}^{m}(\langle \xi^\prime,xL^*_\ell L_\ell\xi \rangle - 2 \langle L_\ell\xi^\prime,x L_\ell\xi \rangle +\langle L^*\ell L_\ell\xi^\prime,x\xi \rangle). \end{split}\end{equation}

This is a natural way to make sense of a Gorini, Kossakowski, Lindblad-Sudarshan
(GKLS) representation of the generator in a generalized form since operators $L_\ell$, $H$ are unbounded.

The following result ensures that the form generator in Eq. \eqref{eq:fg} generates a quantum Markov semigroup and provides its action on Weyl operators; its proof can
be found in \cite{AFP22},
Theorem 2 in Appendix A and Theorem 2.4

\begin{theo}
There exists a unique quantum Markov semigroup, $\TT$ such that, for all $x \in B(\hh)$ and $\xi, \xi^\prime \in D$, the function $t \mapsto \langle \xi^\prime, \TT_t(x) \xi \rangle$ is differentiable and
$$
\frac{d}{dt} \langle \xi^\prime, \TT_t(x) \xi \rangle = \mathfrak{L}(\TT_t(x))[\xi^\prime,\xi], \quad  \forall t \geq 0.
$$
Moreover,
\begin{equation} \label{eq:aW}
\TT_t(W(z))=\exp\left (-\frac{1}{2}\int_0^t \langle e^{sZ}z,Ce^{sZ}z \rangle_\RR ds + i \int_0^t \langle \zeta,e^{sZ}z \rangle_\RR ds\right )W(e^{tZ}z),
\end{equation}
where
\begin{align}
&Zz=[(U^T\overline{U}-V^T\overline{V})/2+i \Omega]z+ [(U^TV-V^TU)/2+i \kappa]\overline{z},\label{eq:Z}\\
&Cz=(U^T\overline{U}+V^T\overline{V})z + (U^TV+V^TU)\overline{z} \label{eq:C}.
\end{align}
\end{theo}
If $\rho$ is the quantum Gaussian state with mean $m$ and covariance $S$, Eq. \eqref{eq:aW} shows that
$\TT_{t*}(\rho)$ is the quantum Gaussian state with mean $m_t$ and covariance $S_t$ given by
$$m_t=e^{tZ^\sharp}m-\int_0^t e^{sZ^\sharp}\zeta ds, \quad S_t=e^{tZ^\sharp}Se^{tZ}+ \int_0^t e^{sZ^\sharp}Ce^{sZ}ds.$$
One can see that the complexifications of $Z$ and $C$ satisfy the following inequality:
\begin{equation} \label{eq:cz}
\mathbf{C} + i(\mathbf{Z}^* \mathbf{J}+\mathbf{J} \mathbf{Z}) \geq 0.
\end{equation}
Such a constraint ensures that the uncertainty principle for field operators is not violated and that $S_t$ is an admissible covariance for every $t \geq 0$.

For the majority of new results in this paper, we will make use of the following assumption:
\begin{equation} \tag{H1} \label{eq:hypo1}
  \text{\textbf{there exists a normal invariant state for $\TT$}.}
\end{equation}
We remark that we do not ask for the invariant state to be unique, nor faithful. \eqref{eq:hypo1} has the following consequence on the spectrum of $\mathbf{Z}$ (Proposition 8 in \cite{FP24}):
\begin{equation} \tag{H2} \label{eq:hypo2}
{\rm Sp}(\mathbf{Z}) \subseteq \{z \in \CC: \Re(z) \leq 0\}.
\end{equation}

We remark that \eqref{eq:hypo1} is strictly stronger than \eqref{eq:hypo2}: for instance if one consider the dynamic given by $W(t)^*\cdot W(t)$ on the one-mode Fock space is such that $Z=0$, but there exists no invariant state.

An important subclass of GQMSs are $^*$-automorphic/purely Hamiltonian dynamics: this happens if and only if one of the following equivalent conditions holds
\begin{itemize}
    \item there are no jump operators in the generator,
    \item $\mathbf{C}=0$.
\end{itemize}
In this case, Eq. \eqref{eq:cz} is equivalent to condition \eqref{eq:symp2}, or in other words $\mathbf{Z}$ is in the symplectic Lie algebra.

There exists two changes of basis that preserve the Gaussianity of quantum Markov semigroups. Given a symplectic transformation $M$, let $\TT^M$ be the quantum Markov semigroup defined as
\begin{equation}\label{eq:meta}
\TT_t^M:=U(M)\TT_t(U(M)^* \cdot U(M))U(M)^*, \quad \forall t \geq 0.\end{equation}
Notice that for every $z \in \CC^d$
$$\TT_t^M(W(z))=e^{-\frac{1}{2}\int_0^t \langle  e^{sZ^M} z, C^M e^{sZ^M } z \rangle_{\RR}ds + i \int_0^t \langle \zeta^M, e^{sZ^M} z \rangle_{\RR}ds}W( e^{tZ^M } z),$$
where
$$Z^M:=MZM^{-1}, \quad C^M:=M^{-\sharp} C M^{-1}, \quad \zeta^M:=M^{-\sharp} \zeta.
$$
It will be useful to introduce the following equivalence relation between GQMSs: $\TT \sim \TT^\prime$ if and only if there exists a symplectic transformation $M$ such that $\TT^\prime=\TT^M$. We will denote by $[\TT]$ the equivalence class of $\TT$. A trivial remark is that any representative $\TT^\prime \in [\TT]$ has a normal invariant state if and only if $\TT$ does: in fact, there is a bijection between the set of invariant states of $\TT$ and $\TT^\prime$. 

Given $w \in \CC^d$, we can define $\TT^{(w)}$ as the following quantum Markov semigroup:
\begin{equation}\label{eq:conjweyl}
\TT^{(w)}_t(\cdot)=W(w)\TT_t(W(w)^* \cdot W(w))W(w)^*, \quad t \geq 0.
\end{equation}
One can check that $\TT^{(w)}$ is again a GQMS with parameters
$$Z^{(w)}=Z, \quad C^{(w)}=C, \quad \zeta^{(m)}=\zeta-2Z^\sharp w.$$

\section{Decoherence-free subalgebra and ${\rm Sp}(\mathbf{Z})$} \label{sec:df}

Firstly, we will briefly recall the notion of decoherence-free subalgebra $\NN$ and its main properties for general quantum Markov semigroups. We will, then, report some known results (\cite{AFP22}) and derive some new ones in the case of GQMSs. Finally, we will draw the connection between $\NN$ and the spectrum of $\mathbf{Z}$ when $\TT$ admits a normal invariant state and prove that, in this case, it is a factor of type $I$.

\bigskip We recall that the decoherence-free subalgebra is defined as
$$\NN:=\{x \in B(\hh): \TT_t(x^*)\TT_t(x)=\TT_t(x^*x), \,\TT_t(x)\TT_t(x^*)=\TT_t(xx^*) \text{ for all }t \geq 0\}.
$$
$\NN$ is the biggest $W^*$-algebra which is invariant for $\TT$ and on which $\TT$ acts as a semigroup of *-endomorphisms (see the proof of Theorem 3.1 in \cite{Ev77}). In general models, there are two cases in which we know that $\TT$ acts as a semigroup of *-automorphisms:
\begin{itemize}
\item when $\TT$ is uniformly continuous (\cite[Theorem 3]{FSU19}) and
\item when $\TT$ admits a normal faithful invariant state (\cite[Lemma 3.4]{He11}).
\end{itemize}
In the case of Gaussian semigroups, we can show that $\TT$ always acts as a group of *-automorphisms on $\NN$. In order to do so, it is useful to recall the following facts about the decoherence-free subalgebra (see Theorem 6, Theorem 13 and Corollary 14 in \cite{AFP22}).

\begin{prop} \label{prop:prevNN}
Let $V$ be the biggest real linear subspace of $\ker(C)$ which is $Z$ invariant, then
\begin{equation} \label{eq:Vinv}
\NN=\{W(z): \, z \in V\}^{\prime\prime}.
\end{equation}
Therefore,
\begin{equation} \label{eq:nnW}
  \TT_t(W(z))=\exp\left ( i \int_0^t \langle \zeta,e^{sZ}z \rangle_\RR ds\right )W(e^{tZ}z), \quad z \in V.  
\end{equation}
Moreover, there exists a pair of natural numbers $d_c, d_f\leq d$ such that
\begin{equation}\label{eq:iso}
\NN\simeq L^\infty(\RR^{d_c};\CC) \overline{\otimes}B(\Gamma(\CC^{d_f}))\end{equation}
and $\TT_t(x)=e^{-iHt}xe^{iHt}$ for every $x \in \NN$, where $H$ is as in Eq. \eqref{eq:hami}.
\end{prop}
The symbol $\simeq$ in Eq. \eqref{eq:iso} means that the two $W^*$-algebras are spatially isomorphic, i.e. that the second can be obtained from the first one via unitary conjugation and viceversa. Notice that the previous result states that $\NN$ is of type $I$; we will use $Z(\NN)$ to denote the center of $\NN$, which in the identification in Eq. \eqref{eq:iso} corresponds to $L^\infty(\RR^{d_c};\CC) \otimes\mathbb{1}$. We are now ready to present our first result.
\begin{prop} \label{prop:auto}
The following statements are true:
\begin{enumerate}
\item $\TT$ acts on $\NN$ as a group of *-automorphisms;
\item $Z(\NN)$ is an invariant $W^*$-algebra for $\TT$ and the action of $\TT$ restricted to $Z(\NN)$ is the one induced by the flow on $\mathbb{R}^{d_c}$ given by a deterministic differential equation of the form
    \begin{equation} \label{eq:cflow}
    \frac{dX_t}{dt}=A X_t +b, \quad A \in M_{d_c}(\RR), \,b \in \RR^{d_c}. 
    \end{equation}
\end{enumerate}  
\end{prop}
The relation between $A$, $b$ and $Z$, $\zeta$ can be understood reading the proof of Proposition \ref{prop:auto}.
\begin{proof}
1. From the fact that for all $x \in \NN$, $\TT_t(x)=e^{-itH}xe^{itH}$ (Proposition \ref{prop:prevNN}), one can already see that it is a semigroup of injective *-endomorphisms, therefore we only need to prove surjectivity.

Using the explicit action of $\TT$ on Weyl operators contained in $\NN$ (Eq. \eqref{eq:nnW}) and the representation of $\NN$ given in Eq. \eqref{eq:Vinv}, one has that
\[\begin{split}\TT_t(\NN)&=\{W(e^{tZ}z):z \in V\}^{\prime\prime}=\{W(z):z \in e^{tZ}(V)\}^{\prime\prime}\\
&=\{W(z):z \in V\}^{\prime\prime}=\NN.\end{split}\]
The last inequality is due to the fact that $V$ is $Z$-invariant and that $\ker(e^{tZ})=\{0\}$. This implies that $\TT_t$ is surjective, hence it is a *-automorphism.

2. In case $\TT_t$ acts as a *-automorphism on $\NN$, it is easy to see that for every $W^*$-subalgebra ${\cal M} \subseteq \NN$, one has
\begin{equation}\label{eq:comm}
\TT_t({\cal M}^\prime\cap \NN) =\TT_t({\cal M})^{\prime}\cap \NN,
\end{equation}
Therefore, it follows easily that $\TT_t(Z(\NN))=Z(\NN)$:
indeed,
\[\begin{split}
\TT_t(Z(\NN))&=\TT_t( \NN^{\prime}\cap \NN)=\TT_t(\NN)^\prime \cap \NN\\
&=\NN^\prime \cap \NN=Z(\NN).
\end{split}\]
The proof of Theorem 13 in \cite{AFP22} shows that, at the cost of passing to another representative $\TT^\prime \in [\TT]$, we can assume that Weyl operators belonging to $Z(\NN)$ correspond to $z \in {\rm span}_{\RR}\{f_1,\dots,f_{d_c}\}$ (where $\{f_1,\dots f_{d}\}$ is the canonical basis of $\CC^d$).
In the representation $Z(\NN) \simeq L^\infty(\RR^{d_c}) \otimes \mathbb{1}$, Weyl operators in $Z(\NN)$ correspond to functions of the form
\[
f_{\tilde{z}}(x)=\exp\left (i\langle x, \tilde{z} \rangle \right ), \quad \tilde{z}=(z_1,\dots, z_{d_c}) \in \CC^{d_c}.
\]
Let us introduce
\[A=(\langle Z f_i, f_j \rangle), \quad b=(\Re( \zeta_j)), \quad i,j=1,\dots, d_c.
\]
Using Eq. \eqref{eq:nnW}, one can easily see that their evolution under $\TT$ is given by
\[
\begin{split}
    f_{\tilde{z}}(x,t)&=\exp \left ( i \left \langle e^{tA}x+\int_0^t e^{sA} bds,\tilde{z} \right \rangle \right )=f_{\tilde{z}}\left ( e^{tA}x+\int_0^t e^{sA} bds\right ).
\end{split}
\]
Eq. \eqref{eq:cflow} follows easily.
\end{proof}

\begin{rem}
There is an alternative proof in order to show that $Z(\NN)$, which can be equivalently characterized as
$$Z(\NN)=\{W(z):\, z \in V, \sigma( z, w  )=0, \, \forall w \in V\}^{\prime\prime},$$
is a $\TT$-invariant $W^*$-algebra. For every $t \geq 0$, there exists a symplectic transformation $M_t$ such that $M_t(V) = V$ and $e^{tZ}_{|V}=M_{t|V}$; therefore, considering $z \in V$ such that $W(z) \in {\cal Z}(\NN)$, one has
$$
\sigma( e^{tZ}z,w  )=\sigma( M_t z, w )= \sigma( z, M^{-1}_t w )=0, \quad \forall w \in V.$$
\end{rem}

An important splitting of the complex plane in order to study the asymptotics of the semigroup generated by $\mathbf{Z}$ is the following one:
\begin{equation} \label{eq:ccsplit}
\mathbb{\CC}=\{\Re(z) <0\}\cup \{\Re(z)=0\} \cup \{\Re(z)>0\}.
\end{equation}
If ${\rm Sp}(\mathbf{Z}) \subseteq \{\Re(z) <0\}$, the semigroup is stable and Theorem 9 in \cite{FP24} shows that $\TT$ admits a unique normal invariant state $\rho_\infty$ which is Gaussian and such that for every initial state $\rho$, one has
$$\lim_{t \rightarrow +\infty} \TT_{*t}(\rho)=\rho_\infty$$
in trace norm; in this case, one has that $\NN=\CC\mathbf{1}$. It is natural to wonder how the spectrum of $\mathbf{Z}$ restricted to ${\cal V}$ (which is the real subspace of $\mathbb{R}^{2d}$ corresponding to $V$ appearing in Proposition \ref{prop:prevNN}) locates in the complex plane with respect to the splitting in Eq. \eqref{eq:ccsplit}; in general it can be in any of the three subsets. However, the existence of a normal invariant state for the semigroup forces it to be only on the imaginary axis, as we will show in the following.

First of all, we prove a simple lemma about purely imaginary eigenvalues of $\mathbf{Z}$.

\begin{lemma} \label{lem:nonex}
If \eqref{eq:hypo1} holds true, the geometric and algebraic multiplicity of the purely imaginary eigenvalues of $\mathbf{Z}$ as an operator on $\CC^{2d}$ coincide.
\end{lemma}
\begin{proof}
We will denote by $\omega_\infty$ any normal invariant state. Let us consider $\lambda \in {\rm Sp}(\mathbf{Z})\cap \{\Re(z)=0\}$.

If $\lambda=0$, then there exists a corresponding eigenvector $x=(x_1,x_2) \in \RR^{2d}$. By contradiction, suppose that there exists $y=(y_1,y_2) \in \RR^{2d}$ such that $\mathbf{Z}y=x$; we can choose $x$ and $y$ with real entries because $\mathbf{Z}$ has real entries as well. Notice that $e^{t\mathbf{Z}}y=y+tx$, which implies that for every $u \in \RR\setminus\{0\}$
\[
\begin{split}
|\omega_\infty(W(u(y_1+iy_2)))|&=|\omega_\infty(\TT_t(W(u(y_1+iy_2)))|\\
&\leq \omega_\infty(W(u(y_1+iy_2+t(x_1+ix_2)))) \xrightarrow[t \rightarrow +\infty]{}0.
\end{split}\]
The limit is due to quantum Riemann-Lebesgue Lemma (\cite[Lemma 7]{FP24}). This shows that $\hat{\omega}_\infty(u(y_1+iy_2))=\delta_0(u)$, which is not the characteristic function of any probability measure on the real line.

If $\lambda=\alpha i$ with $\alpha \neq 0$ and the algebraic and geometric multiplicity of $\lambda$ do not coincide, one can find $w$, $y \in \CC^{2d}$ such that
$$
\mathbf{Z}w=\alpha i w, \quad \mathbf{Z}y=\alpha i y + w.$$
Since $\mathbf{Z}$ has real entries, one has
$$\mathbf{Z}\overline{w}=-\alpha i \overline{w}, \quad \mathbf{Z}\overline{y}=-\alpha i \overline{y} + \overline{w}$$
and, consequently, one has that 
$$
    e^{t\mathbf{Z}}y=e^{i \alpha t}y+te^{i \alpha t }w, \quad e^{t\mathbf{Z}}\overline{y}=e^{-i \alpha t}\overline{y}+te^{-i \alpha t }\overline{w}.
$$
The matrix representation of $e^{t\mathbf{Z}}$ with respect to $\Re(y)$, $\Im(y)$, $\Re(w)$ and $\Im(w)$ is given by
$$
\begin{pmatrix}
    \cos(\alpha t) & \sin(\alpha t)  & t\cos(\alpha t) & t \sin(\alpha t)\\
    - \sin(\alpha t) & \cos(\alpha t) &-t\sin(\alpha t)&  t\cos(\alpha t) \\
    0 &0 &\cos(\alpha t) & \sin(\alpha t)\\
    0 & 0 & -\sin(\alpha t) & \cos(\alpha t)
\end{pmatrix}.
$$
In this case as well, one can see, for instance, that $\max\{\|e^{t\mathbf{Z}}\Re(y)\|, \,\|e^{t\mathbf{Z}}\Im(y)\|\} \rightarrow +\infty$ for $t \rightarrow +\infty$ and one arrives to a contradiction as in the case $\lambda =0$.
\end{proof}

Since $\mathbf{Z}$ has real entries, given an eigenvalue $\lambda$, the complex conjugate $\overline{\lambda}$ is an eigenvalue as well; the corresponding eigenvectors can be chosen as one the conjugate (entrywise) of the other. Therefore, we can always pick a basis of vectors with real entries for the direct sum of the eigenspaces corresponding to $\lambda$ and its conjugate. Given the eigenspace in $\CC^{2d}$ corresponding to the eigenvalues $\lambda$ (which we take as a representative and we assume to have non-negative imaginary part) and $\overline{\lambda}$, we denote by $z_{\lambda,1},\dots, z_{\lambda,k_{\lambda}}$ a possible choice for such a basis. Let us introduce the following notation:
$$
{\cal V}^{\lambda}:={\rm span}_{\mathbb{R}}\{z_{\lambda,1}, \dots, z_{\lambda,k_\lambda}\} \subseteq \RR^{2d},$$
$$
{\cal V}_{0}:=\bigoplus_{\lambda \in {\rm Sp}(\mathbf{Z}) \cap i \mathbb{R}_{\geq 0}} {\cal V}^{\lambda} \subseteq \RR^{2d}.$$
Thanks to Lemma \ref{lem:nonex}, one has that the action of $e^{t\mathbf{Z}}$ on ${\cal V}_{0}$ can be decomposed in $2\times 2$ blocks which are either $\mathbf{1}$ or similar to planar rotations. As usual, we will remove the calligraphic font for the corresponding sets in $\CC^d$, for instance
$$
V_{0}:=\left \{z \in \CC^d: \, \begin{pmatrix} \Re(z)\\ \Im(z) \end{pmatrix} \in {\cal V}_{0}\right \}. $$
\begin{theo} \label{prop:perif}
   If \eqref{eq:hypo1} holds, then $V_{0} \subseteq \ker(C)$ and
    $$\NN=\{W(z):z \in V_{0}\}^{\prime\prime}.
    $$
\end{theo}
\begin{proof}
First of all, we will show that $V_{0} \subseteq \ker(C)$. Assume that ${\rm supp}(C)$ is not orthogonal (with respect to $\langle\cdot , \cdot \rangle_\RR$) to $V_{0}$; then, there must be at least one $\lambda=i\theta$ such that $V^{\lambda}$ is not orthogonal to ${\rm supp}(C)$.

Let us first assume that $\lambda =0$; then, there exists $w \in V_0$ such that $Zw=0$ and $w$ is not orthogonal to ${\rm supp}(C)$. Therefore, for every $u \in \mathbb{R}$ one has
$$
\TT_t(W(uw))=\exp\left (-\frac{tu^2}{2}\langle w,Cw \rangle_\RR + iu t \langle \zeta,w \rangle_\RR\right )W(uw).$$
From the expression above, one can see that
\begin{equation} \label{eq:lim}
\lim_{t \rightarrow +\infty}\TT_t(W(uw)) \xrightarrow{\| \cdot \|_\infty} \delta_{0}(u) \mathbf{1},\end{equation}
which contradicts the fact that $\TT$ admits a normal invariant state: indeed, denoting by $\omega_\infty$ a normal invariant state, Eq. \eqref{eq:lim} implies that $\hat{\omega}_\infty(u w)=\delta_0(u)$, which is not the characteristic function of any probability measure on the real line.

Let us now assume that $\lambda=i\theta \neq 0$; then there exist two linearly independent $w,z \in V_\lambda$ such that $e^{tZ}$ on ${\rm span}_{\mathbb{R}}\{w,z\}$ acts as
$$
\begin{pmatrix} \cos(t\theta) & \sin(t\theta) \\
-\sin(t\theta) & \cos(t\theta)\end{pmatrix}$$
and there exists a vector $a \in {\rm span}_{\mathbb{R}}\{w,z\}$ which is not orthogonal to ${\rm supp}(C)$. Therefore one has that in this case as well
$$
\lim_{t \rightarrow +\infty}\int_0^t \Re( \langle e^{sZ}a, C e^{sZ} a \rangle)ds=  +\infty$$
and for every $u \in \mathbb{R}$
$$\lim_{t \rightarrow +\infty}\TT_t(W(ua)) \xrightarrow{\| \cdot \|_\infty} \delta_{0}(u) \mathbf{1}.$$

\bigskip Let us now prove that the only invariant subspaces for $Z$ in $\ker(C)$ correspond to purely imaginary eigenvalues. Since $\TT$ admits a normal invariant state, we know that ${\rm Sp}(\mathbf{Z}) \subset \{ \Re(z) \leq 0\}$, so it is enough to show that there cannot be any non-trivial $Z$-invariant subspace $W$ such that
$$W \subseteq \ker(C),\quad  {\rm Sp}(\mathbf{Z}_{|{\cal W}}) \subset \{z \in \CC : \, \Re(z)<0\}. $$

Let us consider $w \in W$. Notice that
$$\lim_{t \rightarrow +\infty}e^{tZ}w = 0$$
and
$$
\int_0^{+\infty} e^{tZ}w dt=-Z^{-1}w$$
is well defined. Therefore, if we denote by $\omega_\infty$ any normal invariant state for $\TT$, one has that the following statement holds for every $u \in \RR$:
\[
\begin{split}\omega_\infty(W(uw))&=\omega_\infty(\TT_t(W(uw)))\\
&=\exp \left (iu \int_0^t \langle \zeta,e^{sZ} w \rangle _\RR ds \right )\omega_\infty(W(e^{tZ}w)) \xrightarrow[t \rightarrow +\infty]{} \exp \left (-iu \langle \zeta,Z^{-1}w \rangle_\RR \right ).
\end{split}\]
If $w \neq 0$, this means that the characteristic function of the field operator $R(w)$ in the state $\omega_{\infty}$ is the one of $\delta_{\langle \zeta, Z^{-1}w \rangle_\RR}.$ However, this contradicts the Heisenberg uncertainty principle, therefore $w=0$ and we are done. 
\end{proof}

\begin{rem}
Summing up, Lemma \ref{lem:nonex} and Theorem \ref{prop:perif} show that $Z$ seen as an operator acting on $V_0$ generates a group which is similar to a group of special orthogonal transformations, i.e. there exists a linear transformation $A$ such that $A(V_0)=V_0$ and for all $t \in \RR$
$$e^{tZ^{A\sharp} }_{|V_0}e^{tZ^A}_{|V_0}=\mathbf{1}, \quad  \det(e^{tZ^A}_{|V_0})\equiv 1$$
or, equivalently,
$$Z^{A\sharp}_{|V_0}+Z^A_{|V_0}=0, \quad \tr(Z^A_{|V_0})=0,$$
where $Z^A=AZ A^{-1}$.
\end{rem}
If \eqref{eq:hypo1} holds true, there is an important decomposition of $\CC^d$ induced by the spectral structure of $\mathbf{Z}$.

\begin{lemma} \label{lem:so}
Let us assume \eqref{eq:hypo1} and let us define
$$V_-:=\left \{z \in \CC^d:\lim_{t \rightarrow +\infty}e^{tZ}z=0 \right \}.$$

Then, the following statements are true:
    \begin{enumerate}
    \item $V_0$ and $V_-$ are $Z$-invariant subspaces,
    \item $V_0 \cap V_-=\{0\}$,
    \item $V_0 \oplus V_-=\CC^d,$
    \item $V_0$ and $V_-$ are symplectically orthogonal or, equivalently, $[W(z),W(w)]=0$ for every $z \in V_0$, $w \in V_-$,
    \item The symplectic form restricted to $V_0$ and $V_-$ is nondegenerate.
    \end{enumerate}
    Moreover, for every $z \in \CC^d$, let $z=z_1+z_2$ be the unique decomposition where $z_1 \in V_{0}$ and $z_2 \in V_-$, then
\begin{equation}
{\rm w}^*\text{-}\lim_{t \rightarrow +\infty}\TT_t(W(z))-A(z_2,\infty)\TT_t(W(z_1)=0,
\end{equation}
where
\begin{equation}
A(z_2,\infty)=\exp\left (-\frac{1}{2} \int_0^{+\infty} \langle e^{sZ}z_2, C e^{sZ} z_2 \rangle_\RR ds +i \int_0^{+\infty} \langle \zeta, e^{sZ}z_2 \rangle_\RR ds \right ).
\end{equation}
\end{lemma}
\begin{proof}
1. and 2. can be easily checked using the definitions of $V_0$ and $V_-$.

3. It suffices to show that the real dimensions of ${\cal V}_0$ and ${\cal V}_-$ sum up to $2d$. Let us define ${\cal W}_-$ and ${\cal W}_0$ as the ranges of the spectral projections of $\mathbf{Z}$ corresponding to $\{\Re(z)<0\}$ and $i\RR$, respectively. One can easily see that
$${\cal W}_-:=\left \{z \in \CC^{2d}:\lim_{t \rightarrow +\infty}e^{t\mathbf{Z}}z=0 \right \}.$$
Moreover, we already observed that 
$${\cal W}_0=:{\rm span}_{\CC}{\cal V}_0.$$
The complex dimensions of ${\cal W}_0$ and ${\cal W}_-$ sum up to $2d$. Since $\mathbf{Z}$ has real entries, if $z \in {\cal W}_-$, then $\overline{z} \in {\cal W}_-$ as well; this implies that one can pick a basis for ${\cal W}_-$ composed of vectors with real coefficients, hence the real dimension of
$${\cal V}_-:=\left \{z \in \RR^{2d}:\lim_{t \rightarrow +\infty}e^{t\mathbf{Z}}z=0 \right \}$$
is equal to the complex dimension of ${\cal W}_-$ and we are done.

4. We remark that $V_{0} \subseteq \ker(C)$; since $C$ is positive, one has that $V_{0} \subseteq {\rm rank}(C)^\perp$ as well. Since the symplectic form appears in the commutation relation of the corresponding Weyl operators, it is natural to compare $\TT_t(W(z+w))=e^{i\sigma( z, w )}\TT_t(W(z)W(w))$ and $\TT_t(W(z)W(w))$:
\[
\TT_t(W(z+w))=A(w,t)\exp \left (i \int_0^t \langle \zeta, e^{sZ}z \rangle_\RR ds \right )W(e^{tZ}(z+w))
\]
where
\[
A(w,t)=\exp\left (-\frac{1}{2} \int_0^t \langle e^{sZ}w, C e^{sZ} w \rangle_\RR ds +i \int_0^t \langle \zeta, e^{sZ}w \rangle_\RR ds \right ),
\]
and
\[
\begin{split}\TT_t(W(z)W(w))&=\TT_t(W(z))\TT_t(W(w))\\
&=A(w,t)\exp \left (i \int_0^t \langle \zeta, e^{sZ}z \rangle_\RR ds \right )W(e^{tZ}(z))W(e^{tZ}w)).
\end{split}
\]
The first equality in the previous equation is true thanks to the multiplicative properties of the elements of $\NN$ (see the proof of Theorem 3.1 in \cite{Ev77}). Notice that
\[
{\rm w}^*\text{-}\lim_{t \rightarrow +\infty} W(e^{tZ}(z+w))-W(e^{tZ}z)=0.
\]
Moreover, since $W(e^{tZ}w)$ converges to $\mathbf{1}$ strongly (thanks to the regularity of the Fock representation), then one has that
\[
{\rm w}^*\text{-}\lim_{t \rightarrow +\infty} W(e^{tZ}(z))W(e^{tZ}w))-W(e^{tZ}z)=0
\]
as well. Therefore,
\[\begin{split}
&{\rm w}^*\text{-}\lim_{t \rightarrow +\infty}(e^{i\sigma( z, w )}-1)\TT_t(W(z))\TT_t(W(w))=\\
&{\rm w}^*\text{-}\lim_{t \rightarrow +\infty}\TT_t(W(z+w))-\TT_t(W(z)W(w))=0
\end{split}\]
and, since $\TT_t(W(z))\TT_t(W(w))$ does not converge to $0$ for $t \rightarrow +\infty$, then $\sigma( z, w  )=0$.

5. Let us consider $z \in V_0$ such that $\sigma( z, w )=0$ for all $w \in V_0$; notice that item 4. implies that $\sigma( z, w )=0$ for $w \in V_-$ as well. Since $\CC^d=V_0 \oplus V_-$ and $\sigma (\cdot, \cdot )$ is nondegenerate on $\CC^d$, we can conclude that $z=0$.
\end{proof}

The following corollary is an easy consequence of item 5. in the previous Lemma.

\begin{coro} \label{coro:Ifactor}
If \eqref{eq:hypo1} holds true, then $\NN$ is a factor of type $I$.
\end{coro}

\section{Characterization of GQMSs with normal invariant states} \label{sec:char}

In this section we will provide a characterization of those GQMSs admitting a normal invariant state in terms of $Z$, $C$ and $\zeta$ (or, equivalently, $H$ and $L_l$'s appearing in the form generator). We will start assuming that $\TT$ admits a normal invariant state and derive several consequences of this assumption up to the point where such properties will also be sufficient for $\TT$ to have a normal invariant state.

First of all, we show that if \eqref{eq:hypo1} holds, then there is a representative $\widetilde{\TT}\in [\TT]$ for which the corresponding $\widetilde{Z}$ and $\widetilde{C}$ have a simple form. We can reduce to study $\widetilde{\TT}$ keeping in mind that every result we derive can be translated in terms of $\TT$ conjugating every operator using the unitary operator connecting $\TT$ to $\widetilde{\TT}$.

\begin{lemma} \label{lem:scb}
    If \eqref{eq:hypo1} holds true, there exists a symplectic transformation $M$ such that
    \begin{align*}
        &V^M_{0}:=M(V_{0})={\rm span}_\RR\{e_1,\dots, e_{d_0}, ie_1,\dots, ie_{d_0}\}, \\
        &V^M_-:= M(V_-)={\rm span}_\RR\{e_{d_0+1},\dots, e_{d},ie_{d_0+1},\dots, ie_{d}\},
        \end{align*}
    where $d_0={\dim}_\CC(V_0)$ and such that $Z^M:=MZM^{-1}$ restricted to $V_0^M$ reads as (when one picks the basis $e_1,\dots, e_{d_0}, ie_1,\dots, ie_{d_0}$)
    \[
    \begin{pmatrix} O_{d_0} & -\Phi \\
    \Phi & O_{d_0}
    \end{pmatrix},
    \]
    where $\Phi$ is a diagonal matrix with diagonal entries $\phi_1,\dots, \phi_{d_0} \in \RR$. Therefore,
    $$(Z^{M}_{|V_0^M})^{\sharp} + Z^{M}_{|V_0^M}=0 \text{ and } \tr(Z^M_{|V_0^M})=0.$$
\end{lemma}

We remark that $M$ in Lemma \ref{lem:scb} is not unique, however in the following proof one can find an explicit recipe for building such an $M$, given $Z$.

\begin{proof}

We showed that there exist $\phi_1, \dots, \phi_{d_0} \in \RR$ and $w_1,\dots, w_{d_0} \in \CC^{2d}$ such that
$$\mathbf{Z}w_j=i\phi_j w_j, \quad \mathbf{Z}\overline{w_j}=-i\phi_j \overline{w_j}$$
and ${\cal V}_0={\rm span}_\RR\{\Re(w_j), \Im(w_j)\}$. We recall that
$${\cal V}_0=\left \{ \begin{pmatrix} x\\ y \end{pmatrix} \in \RR^{2d}: x+iy \in V_0\right \}.$$

Notice that
$$\mathbf{Z}\Re(w_j)=-\phi_j\Im( w_j), \quad \mathbf{Z}\Im(w_j)=\phi_j \Re(w_j).$$
Once we prove that we can pick $w_j$'s such that for every $j,k=1,\dots, d_0$ 
$$\langle \Re(w_j), \mathbf{J} \Im(w_k) \rangle =\delta_{jk} \quad \langle \Re(w_j), \mathbf{J} \Re(w_k) \rangle=0, \quad \langle \Im(w_j), \mathbf{J} \Im(w_k) \rangle=0,$$
we can define $\mathbf{M}$ such that
$\mathbf{M}\Re(w_j)=e_j, \quad \mathbf{M}\Im(w_j)=ie_j$
and extend it to any symplectic transformation to the whole $\RR^{2d}$ using Proposition 12, item 1. in \cite{AFP22}.

First of all, let us consider $\phi_j \neq \phi_k$, then for every $x \in {\rm span}_\RR\{\Re(w_j), \Im(w_j)\}$, $y \in {\rm span}_\RR\{\Re(w_k), \Im(w_k)\}$ one has 
\[-\phi_j^2 \langle x, \mathbf{J} y \rangle = \langle \mathbf{Z}^2 x, \mathbf{J} y \rangle= \langle  x, \mathbf{J} \mathbf{Z}^2y \rangle=-\phi_k^2 \langle x, \mathbf{J} y \rangle,\]
hence $\langle x, \mathbf{J} y \rangle=0$.

In other words, subspaces corresponding to different angles are symplectically orthogonal and $\langle \cdot, \mathbf{J} \cdot \rangle$ is non-degenerate when restricted to any susbpace corresponding to a fixed angle $\phi$; we can immediately deduce that we can find a symplectic basis for the kernel of $\mathbf{Z}$, i.e. the subspace corresponding to $\phi_j$'s equal to $0$.

Let us now focus on the subspaces corresponding to non-zero angles: let us consider all the $w_{j_1}, \dots, w_{j_n}$ corresponding to $\phi_j$'s equal to a certain $\phi \neq 0$. We will prove the statement by induction: if $n=1$, since the symplectic form is non-degenerate when restricted to ${\rm span}_\RR\{\Re(w_{j_1}), \Im(w_{j_1})\}$, one has that $\langle\Re(w_{j_1}), \mathbf{J}\Im(w_{j_1}) \rangle \neq 0$ and we can pick real multiples of $\Re(w_{j_1})$ and $\Im(w_{j_1})$ which are a symplectic basis.

Now let us consider $n \geq 2$. First of all, notice that for every $l,k=1, \dots, n$, one has $\langle w_{j_l}, \mathbf{J} \overline{w_{j_k}} \rangle=0$: indeed,
$$
-i\phi \langle w_{j_l}, \mathbf{J} \overline{w_{j_k}} \rangle=\langle \mathbf{Z}w_{j_l}, \mathbf{J} \overline{w_{j_k}} \rangle=-\langle w_{j_l}, \mathbf{J} \mathbf{Z}\overline{w_{j_k}} \rangle=i\phi \langle w_{j_l}, \mathbf{J} \overline{w_{j_k}} \rangle,$$
where we used that on ${\cal V}_0$ one has $\mathbf{Z}^*\mathbf{J}+\mathbf{J}\mathbf{Z}=0$. This implies that $\langle w_{j_l}, \mathbf{J} \overline{w_{j_k}} \rangle=0$. In terms of $\Re(w_{j})$'s and $\Im(w_j)$'s, this translates as
\begin{equation} \label{eq:lemmaM1}
    \begin{split}
        &\langle \Re(w_{j_l}), \mathbf{J} \Re(w_{j_k}) \rangle = \langle \Im(w_{j_l}), \mathbf{J} \Im(w_{j_k}) \rangle, \\
        &\langle \Re(w_{j_l}), \mathbf{J} \Im(w_{j_k}) \rangle = \langle \Re(w_{j_k}), \mathbf{J} \Im(w_{j_l}) \rangle.
    \end{split}
\end{equation}

We claim that there exists $v \in {\rm span}_\CC \{w_{j_1}, \dots, w_{j_n}\}$ such that
\begin{equation} \label{eq:lemmaM2}
\langle \Re(v), \mathbf{J} \Im(v) \rangle =1.
\end{equation}
In this case, we can choose $\tilde{w}_{j_1}, \dots, \tilde{w}_{j_{n-1}} \in {\rm span}_\CC \{w_{j_1}, \dots, w_{j_n}\}$ such that $\{\Re(\tilde{w}_{j_l}), \Im(\tilde{w}_{j_l})\}$ is symplectically orthogonal to $\{\Re(v), \Im(v)\}$. Notice that, using Eq. \eqref{eq:lemmaM1}, we can reduce to check only that $\{\Re(\tilde{w}_{j_l})\}$ is symplectically orthogonal to $\{\Re(v), \Im(v)\}$ and this is always possible (every vector in ${\rm span}_{\RR}\{\Re(w_{j_i}), \Im(w_{j_i}): \, i=1,\dots n\}$ can be expressed as the real part of a vector in ${\rm span}_{\mathbb{C}}\{w_{j_1}, \dots, w_{j_n}\}$). Therefore we can write
$${\rm span}_{\RR}\{\Re(v), \Im(v)\} \oplus {\rm span}_{\RR}\{\Re(\tilde{w}_{j_l}), \Im(\tilde{w}_{j_l})\}_{l=1}^{n-1}$$
and we can use the inductive step.

Let us show that we can find $v$ that satisfies Eq. \eqref{eq:lemmaM2}: by contradiction, suppose that there is not such a $v$. Let us write $v=\sum_{l=1} (\alpha_l + i \beta_l) w_{j_l}$ for $\alpha_l, \beta_l \in \RR$, then one can easily see that
$$\Re(v)=\sum_{l=1}^{n} \alpha_l \Re(w_{j_l})-\beta_l \Im(w_{j_l}), \quad \Im(v)=\sum_{l=1}^{n} \alpha_l \Im(w_{j_l})+\beta_l \Re(w_{j_l}).$$
Therefore, for every $\alpha_l, \beta_l \in \RR$ one has
\[
\begin{split}
\langle \Re(v), \mathbf{J} \Im(v) \rangle &=\sum_{l=1}^{n} (\alpha_l^2+\beta_l^2) \langle \Re(w_{j_l}), \mathbf{J} \Im(w_{j_l}) \rangle \\
&+\sum_{l\neq k=1}^{n} (\alpha_l \beta_k - \beta_l \alpha_k) \langle \Re(w_{j_l}), \mathbf{J} \Re(w_{j_k}) \rangle\\
&+ \sum_{l\neq k=1}^{n} (\alpha_l \alpha_k + \beta_l \beta_k) \langle \Re(w_{j_l}), \mathbf{J} \Im(w_{j_k}) \rangle=0,
\end{split}\]
where we used Eq. \eqref{eq:lemmaM1}. However, this can only be true if
\begin{itemize}
\item $\langle \Re(w_{j_l}), \mathbf{J} \Im(w_{j_l}) \rangle=0$ for $l=1,\dots, n$ (picking $\alpha_l=\delta_{lk}, \beta_l\equiv 0$ for $k=1,\dots, d_0$),
\item $\langle \Re(w_{j_l}), \mathbf{J} \Re(w_{j_k}) \rangle=0$ for $l\neq k =1, \dots, n$ (picking $\alpha_l=\delta_{lm}$, $\beta_l=\delta_{lp}$ for $m\neq p =1,\dots, d_0$,
\item $\langle \Re(w_{j_l}), \mathbf{J} \Im(w_{j_k}) \rangle=0$ for $l\neq k=1,\dots, d_0$ (picking $\alpha_l=\delta_{lm}+\delta_{lp}$, $\beta_l \equiv 0$ for $m\neq p=1, \dots, n$).
\end{itemize}
We got to a contradiction, because this would imply that $\langle \cdot, \mathbf{J} \cdot \rangle$ is degenerate on ${\rm span}_{\RR} \{\Re(w_{j_l}), \Im(w_{j_l})\}_{l=1}^{n}$.
\end{proof}

From now on, \textbf{we will denote by $\widetilde{\TT}$ the representative in $[\TT]$ corresponding to the symplectic transformation in Lemma \ref{lem:scb}}; every symbol with a tilde on top will denote the same object, but referred to $\tilde{\TT}$, instead of $\TT$, e.g. $\widetilde{Z}$, $\widetilde{C}$, $\widetilde{V}_0$, $\widetilde{V}_-$ and so on. $\widetilde{P}_0$ (resp. $\widetilde{P}_-$) will denote the ortoghonal projection onto $\widetilde{V}_0=\{e_1,\dots, e_{d_0}, ie_1, \dots, ie_{d_0}\}$ (resp. $\widetilde{V}_-=\{e_{d_0+1},\dots, e_{d}, ie_{d_0+1}, \dots, ie_{d}\}$). We will often make use of the identification
$\FS=\Gamma(\widetilde{V}_{0}) \otimes \Gamma(\widetilde{V}_-)$.

The first result of this section points out the structure of $Z$ and $C$ in case \eqref{eq:hypo1} is true and how this translates in terms of $H$ and $L_l$'s.

\begin{lemma}
  Let $d_0$ be an integer smaller or equal than $d$ and $\Phi$ the diagonal matrix with diagonal entries given by $\phi_1,\dots, \phi_{d_0}$. The following statements are equivalent:
  \begin{enumerate}
      \item in the canonical basis, $\mathbf{Z}$ and $\mathbf{C}$ read as follows 
      $$\mathbf{Z}=\begin{pmatrix} O & -\Phi & O \\
      \Phi & O & O\\
      O & O & \mathbf{Z}_- \end{pmatrix}, \quad \mathbf{C}=\begin{pmatrix} O & O & O\\
      O & O & O\\
      O & O & \mathbf{C}_- \end{pmatrix},$$
      where $O$ are null blocks and the blocks correspond to the subspaces $V^Q:={\rm span}_{\RR}\{f_1,\dots, f_{d_0}\}$, $V^P:={\rm span}_{\RR}\{if_1,\dots, if_{d_0}\}$ and $W:={\rm span}_{\RR}\{f_{d_0+1},\dots, f_{d}, if_{d_0+1},\dots, if_{d}\}$.
      \item $H$ and $L_l$'s introduced in Eq. \eqref{eq:hami} are of the form $H=H_0+H_-$, where
  \begin{align} 
  &H_0:=\sum_{j=1}^{d_0} \phi_j a^\dagger_j a_j + \frac{1}{2}\sum_{j=1}^{d_0} \zeta_j a_j^\dagger + \overline{\zeta}_j a_j,\\ \label{eq:H0}
  &H_-:=\sum_{k,j={d_0+1}}^{d} \left ( \Omega_{jk} a^\dagger_j a_k + \frac{\kappa_{jk}}{2}a_j^\dagger a_k^\dagger+\frac{\overline{\kappa}_{jk}}{2}a_j a_k\right )+\frac{1}{2}\sum_{j=d_0+1}^{d} \zeta_j a_j^\dagger + \overline{\zeta}_j a_j,
  \end{align}
  and
  $$L_l=\sum_{j=d_0+1}^{d} \overline{v}_{lj}a_j + u_{lj} a_j^\dagger.$$
  \end{enumerate}

Moreover, if \eqref{eq:hypo1} is true, then previous conditions hold for $\widetilde{\TT}$. \end{lemma}
\begin{proof}
\textbf{1. is true for $\widetilde{\TT}$ under \eqref{eq:hypo1}.} $\widetilde{\mathbf{Z}}$ is in the form as in 1. due to Lemma \ref{lem:scb} and the fact that both $\widetilde{V}_0=V^Q\oplus V^P$ and $\widetilde{V}_-=W$ are $\widetilde{Z}$-invariant subspaces. Regarding $\widetilde{\mathbf{C}}$, Proposition \ref{prop:perif} implies that $\widetilde{\mathbf{C}}=\widetilde{\mathbf{C}}\widetilde{\mathbf{P}}_-$, hence
$$\widetilde{\mathbf{C}}\widetilde{\mathbf{P}}_-=\widetilde{\mathbf{C}}=\widetilde{\mathbf{C}}^*=\widetilde{\mathbf{P}}_-^*\widetilde{\mathbf{C}}^*= \widetilde{\mathbf{P}}_- \widetilde{\mathbf{C}}$$
as well.

\textbf{1. implies 2.} Since $V:=V^Q\oplus V^P \subseteq \ker{C}$, for $i=1,\dots, d_0$ one has
$$0=Cf_i=(U^T \overline{U}+V^T \overline{V})f_i +(U^T U+V^TV)f_i$$
and
$$0=Cif_i=i(U^T \overline{U}+V^T \overline{V})f_i -i(U^T U+V^TV)f_i.
$$
Since $V$ is closed under multiplication by $-i$, one has that $(U^T \overline{U}+V^T \overline{V})f_i=0$; moreover, 
$$(U^T \overline{U}+V^T \overline{V})f_i=0 \Leftrightarrow \overline{U}f_i=\overline{V}f_i=0 \Leftrightarrow Uf_i=Vf_i=0.$$
Consequently, $\overline{v}_{li}=u_{li}=0$ for every $l=1,\dots, m$ and $i=1,\dots, d_0$ and the statement about jump operators is proved.

Furthermore, one has for every
$i=1,\dots, d_0$
$$
Z f_i= i (\Omega f_i+\kappa f_i),\,  Z if_i= i (\Omega f_i-\kappa f_i) \in V.$$
Using again that $V$ is closed under multiplication for $-i$, it follows that
$$(\Omega f_i+\kappa f_i), \,(\Omega f_i-\kappa f_i) \in V$$
and $\Omega f_i$, $\kappa f_i$ are in $V$ as well, which implies $\Omega_{ij}=\kappa_{ij}=0$ for $i\leq d_0$ and $j>d_0$. Since $\Omega$ is Hermitian and $\kappa$ is symmetric, one gets that $\Omega_{ij}=\kappa_{ij}=0$ for $i> d_0$ and $j\leq d_0$ as well.

\bigskip Let us consider $\Omega^0:=(\Omega_{ij})_{i,j=1,\dots, d_0}$ and $\kappa^0:=(\kappa_{ij})_{i,j=1,\dots, d_0}$ and recall that $\Omega^{0\dagger}=\Omega^0$ and $\kappa^{0T}=\kappa^0$; this implies that
$$\Re(\Omega^0)=\Re(\Omega^0)^T, \, \Im(\Omega^0)=-\Im(\Omega^0)^T, \, \Re(\kappa)=\Re(\kappa)^T, \, \Im(\kappa)=\Im(\kappa)^T.
$$
Equation \eqref{eq:Z} implies that
$Zz=i\Omega^0 z + i \kappa^0 \overline{z}$ and, equivalently,
\[\begin{split} \mathbf{Z}_{|{\cal V}}=\begin{pmatrix} -\Im(\Omega^0)-\Im(\kappa^0) & \Re(\kappa^0)-\Re(\Omega^0) \\
\Re(\Omega^0)+\Re(\kappa^0) & -\Im(\Omega^0)+\Im(\kappa^0)\end{pmatrix} .\end{split}\]
Imposing that $\mathbf{Z}_{|{\cal V}}+\mathbf{Z}_{|{\cal V}}^\sharp=0$ one gets
$$\begin{pmatrix} -\Im(\Omega^0)-\Im(\kappa^0) & \Re(\kappa^0)-\Re(\Omega^0) \\
\Re(\Omega^0)+\Re(\kappa^0) & -\Im(\Omega^0)+\Im(\kappa^0)\end{pmatrix}=\begin{pmatrix} -\Im(\Omega^0)+\Im(\kappa^0) & -\Re(\kappa^0)-\Re(\Omega^0) \\
\Re(\Omega^0)-\Re(\kappa^0) & -\Im(\Omega^0)-\Im(\kappa^0)\end{pmatrix},$$
hence $\kappa^0=0$. Moreover, one has
\[\begin{split}
    \mathbf{Z}_{|{\cal V}}=\begin{pmatrix} -\Im(\Omega^0) & -\Re(\Omega^0) \\
\Re(\Omega^0) & -\Im(\Omega^0)\end{pmatrix}=\begin{pmatrix} O & -\Phi \\
\Phi  & O\end{pmatrix},
\end{split}\] Therefore, $\Omega^0=-\Phi$.

\textbf{2. implies 1.} This is a trivial check using Equations \eqref{eq:Z} and \eqref{eq:C}.

\end{proof}

The last result we are going to prove is a further constraint on $\widetilde{H}_0$.

\begin{lemma} \label{lem:transmodes}
If \eqref{eq:hypo1} is true, then for every $j=1,\dots, d_0$ one has that
\begin{equation}\label{eq:phizeta}\phi_j=0 \text{ implies that } \widetilde{\zeta}_j=0.\end{equation}

Therefore, there exists $w \in \CC^d$ such that $\widetilde{\TT}^{(w)}$ defined as in Eq. \eqref{eq:conjweyl} is such that $\widetilde{\zeta}^{(w)}_j=0$ for $j=1,\dots d_0$.
\end{lemma}
Lemma \ref{lem:transmodes} shows that $\widetilde{H}^{(w)}=\sum_{j=1}^{d_0}\phi_j a_j^\dagger a_j$ and $\widetilde{H}_0=W(w)^*\widetilde{H}_0^{(w)}W(w)$, therefore $\widetilde{H}_0$ admits an orthonormal basis of eigenvectors. 
\begin{proof}
By contradiction, suppose that the statement in Eq. \eqref{eq:phizeta} is false for some index $j$, which, without any loss of generality we can assume to be equal to $1$. Let us consider any normal invariant state $\omega_\infty$ for $\widetilde{\TT}$; for $z=(\widetilde{\zeta}^0_1,0,\dots,0)$, one has
$$\hat{\omega}(uz)=e^{itu}\hat{\omega}(uz),$$
which implies that $\hat{\omega}(uz) \in \{0,1\}$ and we get to a contradiction.

Let us define the following operator
$$\mathbf{A}=\begin{pmatrix} O & \Phi^{-1} & O \\
      \Phi^{-1} & O & O\\
      O & O & O \end{pmatrix},$$
    where the blocks correspond to the decomposition $V^Q$, $V^P$ and $W$ and $\Phi^{-1}$ is the pseudoinverse of $\Phi$. For $w=\frac{1}{2}A\tilde{\zeta}$, one can easily  that $\widetilde{\TT}^{(w)}$ has the claimed property.
\end{proof}

So far, we derived some necessary conditions on $Z$, $C$ and $\zeta$ (or, equivalently, on $H$ and $L_l$'s) for $\TT$ to admit a normal invariant state; on the other hand, it is easy to notice that they are also sufficient. Therefore, we are ready to state the main result of this section.

\begin{theo} \label{thm:main}
The following are equivalent:
\begin{enumerate}
\item $\TT$ admits a normal invariant state;
\item there exist an integer $1 \leq d_0 \leq d$, real numbers $\phi_1,\dots, \phi_{d_0}$ and a representative $\widetilde{\TT} \in [\TT]$ such that
$$\widetilde{\mathbf{Z}}=\begin{pmatrix} \widetilde{\mathbf{Z}}_0 & O\\
      O  & \widetilde{\mathbf{Z}}_- \end{pmatrix}, \quad \widetilde{\mathbf{C}}=\begin{pmatrix} O & O\\
     O & \widetilde{\mathbf{C}}_- \end{pmatrix}, \quad \widetilde{\zeta}=\begin{pmatrix} \widetilde{\zeta}_0 \\ \widetilde{\zeta}_- \end{pmatrix} $$
      where
      \begin{itemize}
          \item the blocks correspond to the splitting of $\CC^d$ given by
          $${\rm span}_{\RR}\{f_1,\dots, f_{d_0},if_1,\dots, if_{d_0}\}\oplus {\rm span}_{\RR}\{f_{d_0+1},\dots, f_{d}, if_{d_0+1}, \dots, if_d\},$$
          \item
          $$\widetilde{\mathbf{Z}}_0=\begin{pmatrix}O & -\Phi \\
      \Phi & O \end{pmatrix}$$ and
      $\Phi$ is the diagonal matrix with diagonal entries $\phi_1, \dots, \phi_{d_0}$,
          \item ${\rm Sp}(\widetilde{\mathbf{Z}}_-) \subseteq \{ \Re(z)<0\} $,
          \item $\widetilde{\zeta}_0 \in {\rm supp}(\widetilde{\mathbf{Z}}_0)$.
      \end{itemize}

      \item There exist an integer $1 \leq d_0 \leq d$, real numbers $\phi_1,\dots, \phi_{d_0}$ and a representative $\widetilde{\TT} \in [\TT]$ such that $\widetilde{H}$ and $\widetilde{L}_l$'s introduced in Eq. \eqref{eq:hami} are of the form $\widetilde{H}=\widetilde{H}_0+\widetilde{H}_-$, where
  \begin{align} 
  &\widetilde{H}_0:=\sum_{j=1}^{d_0} \phi_j a^\dagger_j a_j + \frac{1}{2}\sum_{j=1}^{d_0} \widetilde{\zeta}_j a_j^\dagger + \overline{\widetilde{\zeta}}_j a_j,\\ \label{eq:H0}
  &\widetilde{H}_-:=\sum_{k,j={d_0+1}}^{d} \left ( \widetilde{\Omega}_{jk} a^\dagger_j a_k + \frac{\widetilde{\kappa}_{jk}}{2}a_j^\dagger a_k^\dagger+\frac{\overline{\widetilde{\kappa}}_{jk}}{2}a_j a_k\right )+\frac{1}{2}\sum_{j=d_0+1}^{d} \widetilde{\zeta}_j a_j^\dagger + \overline{\widetilde{\zeta}}_j a_j,
  \end{align}
  and
  $$\widetilde{L}_l=\sum_{j=d_0+1}^{d} \overline{\widetilde{v}}_{lj}a_j + \widetilde{u}_{lj} a_j^\dagger.$$
  Moreover, $\widetilde{\zeta}_j\neq 0$ only if $\phi_j\neq 0$ and ${\rm Sp}(\mathbf{\widetilde{Z}}_-) \subset \{\Re(z)<0\}$, where
  $$\widetilde{Z}_-z=[(\widetilde{U}_-^T\overline{\widetilde{U}_-}-\widetilde{V}_-^T\overline{\widetilde{V}}_-)/2+i \widetilde{\Omega}_-]z+ [(\widetilde{U}_-^T\widetilde{V}-\widetilde{V}^T\widetilde{U}_-)/2+i \widetilde{\kappa}]\overline{z}$$
  and $\widetilde{U}_-=(\widetilde{u}_{lj})$, $\widetilde{V}=(\widetilde{v}_{lj})$ for $l=1,\dots, m$ and $j=d_0+1,\dots, d$.
  \end{enumerate}
\end{theo}
The previous result has an obvious consequence regarding quadratic Hamiltonians.
\begin{coro} \label{coro:quadhami}
A quadratic Hamiltonian admits a ground state if and only if it is unitarily equivalent (via displacement or Bogoliubov transformations) to
$$H=\sum_{j=1}^{d} \phi_j a_j^*a_j + \alpha \mathbb{1}, \quad \phi_1,\dots,\phi_{d_0} \in \RR_{\geq 0}, \alpha \in \RR.$$
\end{coro}

Following a completely different path, one could use Theorem 2.4 in \cite{De17} as well in order to show Corollary \ref{coro:quadhami} in the case $\zeta=0$.

\section{Set of invariant states and irreducibility} \label{sec:stris}

In this section we will make use of the structure theorem obtained in the previous section (Theorem \ref{thm:main}) in order to completely characterize the convex set of normal invariant states for a GQMS $\TT$; we recall that there is no loss in studying the set of normal invariant states of $\tilde{\TT}$ instead, since the two sets are unitarily equivalent. In this section we will make use of the notation in Theorem \ref{thm:main} without recalling it.

\bigskip Let us denote by $\rho_\infty$ the Gaussian state on $\Gamma(\widetilde{V}_-)\simeq \Gamma(\CC^{d-d_0})$ with the following mean and covariance:
\begin{equation} \label{eq:asypar}
m_\infty=\int_0^\infty e^{s \widetilde{Z}_-^\sharp}  \widetilde{\zeta}_- ds, \quad \Sigma_\infty= \int_0^\infty e^{s\widetilde{Z}_-^\sharp} \widetilde{C}_- e^{s \widetilde{Z}_-} ds.\end{equation}

In this section we will provide a complete description of the convex set of normal invariant states for GQMSs. This will allow us to derive some easy criteria in terms of $\mathbf{Z}$ and $\mathbf{C}$ to establish whether $\TT$ admits a faithful invariant state and if it is irreducible.

\begin{theo} \label{theo:invst}

If \eqref{eq:hypo1} holds, then there exists a Weyl operator acting only on the first $d_0$ modes $W:B(\Gamma(\CC^{d_0})) \rightarrow B(\Gamma(\CC^{d_0}))$ such that every normal invariant state for $\tilde{\TT}$ is of the form
$$W^*\omega W \otimes \rho_\infty$$
where $\omega$ is a state that commutes with $\hat{H}:=\sum_{j=1}^{d_0} \phi_j a_j^\dagger a_j$ and $\rho_\infty$ is the quantum Gaussian state on the modes $j=d_0+1,\dots, d$ with parameters defined in Eq. \eqref{eq:asypar}.
\end{theo}
\begin{proof}
Let us consider a normal invariant state $\omega_\infty$ for $\widetilde{\TT}$ and let us define $\omega=\tr_{\Gamma(\widetilde{V}_-)}(\omega_\infty)$; together with Theorem \ref{thm:main}, Lemma \ref{lem:so} shows that for every $z \in \CC^d$
$$\lim_{t \rightarrow +\infty}\hat{\omega}_\infty(z)-\hat{\omega}_t(z_1)\hat{\rho}_\infty(z_2)=0,$$
where
$$\omega_t=e^{it\widetilde{H}_0}\omega e^{-it\widetilde{H}_0}. $$
Therefore $\omega_\infty=\omega \otimes \rho_\infty$ and $\omega$ is a state that commutes with $\widetilde{H}_0=W\hat{H} W^*$ (Lemma \ref{lem:transmodes}), or, equivalently, $W^* \omega W$ commutes with $\hat{H}$ and we are done.
\end{proof}

Notice that $\TT$ fails to have a faithful invariant state if and only if $\rho_\infty$ is not faithful. Thanks to the characterization of faithful Gaussian invariant states in terms of their covariance matrix (see Theorem 4 in \cite{Pa12}) we have the following result.

\begin{coro} \label{coro:faithful}
    If $\eqref{eq:hypo1}$ holds, then $\TT$ admits a faithful normal invariant state if and only if
    $$\mathbf{\Sigma_\infty} +i \mathbf{J}_- >0,$$
where $\mathbf{J}_-$ is the symplectic matrix restricted to $\widetilde{V}_-$.
\end{coro}

Another useful consequence is an easy criterion for the irreducibility of $\TT$. The notion of irreducibility for quantum evolutions traces back to \cite{Da70} and equivalent characterizations and detailed discussions about it can be found in Section 3 of \cite{CP16} or Proposition 5.1 in \cite{FR06}. We recall that $\TT$ is said to be irreducible if there are not nontrivial subharmonic projections, i.e. orthogonal projections $p$ such that
$$\TT_t(p) \geq p, \quad \forall t \geq 0.$$

\begin{coro} \label{coro:irr}
If \eqref{eq:hypo1} is true, then $\TT$ is irreducible if and only if
\begin{equation} \label{eq:irrcond}
\text{there are not nontrivial $\mathbf{Z}$-invariant subspaces in }\ker(\mathbf{C}_Z),
\end{equation}
where $\mathbf{C}_Z:=\mathbf{C}+i(\mathbf{Z}^T\mathbf{J}+\mathbf{J}\mathbf{Z})$.
\end{coro}

As one should expect, condition in Eq. \eqref{eq:irrcond} does not change if we conjugate the semigroup via a metaplectic transformation corresponding to the symplectic transformation $M$: indeed, we recall that
$Z^M=MZM^{-1}$ and $C_Z^M=M^{-\sharp}C_Z M^{-1}$; therefore, $V$ is a $Z$-invariant subspace in $\ker(C_Z)$ if and only if $M(V)$ is a $Z^M$-invariant subspace in $\ker(C^M_Z)$.

\begin{proof}
    Under \eqref{eq:hypo1}, we have the following chain of equivalences:
    \[\begin{split}
        &\TT \text{ is irreducible} \Leftrightarrow \text{ it admits a unique faithful normal invariant state} \Leftrightarrow \\
        & Z \text{ is stable and }\mathbf{\Sigma_\infty}+i\mathbf{J} >0 \Leftrightarrow Z \text{ is stable and } \int_0^{+\infty} e^{s\mathbf{Z}^T}\mathbf{C}e^{s\mathbf{Z}}ds+i\mathbf{J} >0\Leftrightarrow \\
        & Z \text{ is stable and } \int_0^{+\infty} e^{s\mathbf{Z}^T}\mathbf{C}_Ze^{s\mathbf{Z}}ds >0. \qquad (*)
    \end{split}\]
    The last term in the previous chain of equivalent statements is also equivalent to the statement. Indeed, assume that $(*)$ is true and, by contradiction, that there is a nontrivial $\mathbf{Z}$-invariant subspace ${\cal W}$ in $\ker(\mathbf{C}_Z)$, then for every $w \in {\cal W}$ one has that
    $$\int_0^{+\infty} e^{s\mathbf{Z}^T}\mathbf{C}_Ze^{s\mathbf{Z}}wds=0,$$
    which contradicts the fact that $$\int_0^{+\infty} e^{s\mathbf{Z}^T}\mathbf{C}_Ze^{s\mathbf{Z}}ds>0.$$
    On the other hand, let us assume that the statement in Eq. \eqref{eq:irrcond} is true; then ${\cal V}_0$ must be trivial, since it is $\mathbf{Z}$-invariant, $(\mathbf{Z}^T\mathbf{J}+\mathbf{J}\mathbf{Z})_|{{\cal V}_0}=\mathbf{0}$ and ${\cal V}_0 \subseteq \ker{\mathbf{C}}$, hence for every $w \in {\cal V}_0$
    $$\mathbf{C}_Zw = \mathbf{C}w=0.$$
    Therefore $\mathbf{Z}$ is stable, $\int_0^{+\infty} e^{s\mathbf{Z}^T}\mathbf{C}_Ze^{s\mathbf{Z}}ds$ is well defined and it is strictly positive due to \eqref{eq:irrcond}.
\end{proof}

\bigskip We recall that the equivalent condition to irreducibility given in Eq. \eqref{eq:irrcond} has a counterpart for classical systems as well (see Section \ref{sec:cOU}).

The existence of a unique faithful normal invariant state for the semigroup (or equivalently the existence of a normal invariant state and irreducibility of the semigroup) is necessary in order to talk about its spectral gap and it is natural to wonder whether this is also sufficient to ensure that the spectral gap is strictly positive. This is not the case for the GNS embedding, since for having a strictly positive spectral gap one needs $\mathbf{C}_Z>0$ and the example in Sections 6 and 7.1 in \cite{FPSU24} shows that this is not implied by the existence of a unique faithful normal invariant state.

On the other hand, in the case of the KMS embedding, the condition is the following: let us consider the symplectic diagonalization of $\mathbf{\Sigma}_\infty=\mathbf{M}^TD \mathbf{M}$ (Theorem 2 in \cite{Pa12}), then the semigroup has a non-zero KMS spectral gap if and only if
\begin{equation} \label{eq:KMSgap}
-(\mathbf{Z}^T\mathbf{M}^Tf(D) \mathbf{M}+\mathbf{M}^Tf(D) \mathbf{M}\mathbf{Z})>0,\end{equation}
where $f(x)={\rm csch}\, {\rm cotgh}^{-1}(x)$ (see Theorem 26 in \cite{FPSU24}). 
Notice that in general the condition in Eq. \eqref{eq:KMSgap} is stronger than the one for the unique Gaussian state to be faithful, i.e. $\Sigma_\infty+i\mathbf{J}>0$: indeed, \[\begin{split}
    \Sigma_\infty+i\mathbf{J}>0 \Leftrightarrow D + i\mathbf{J}>0 \Leftrightarrow D > \mathbf{1} \Leftrightarrow f(D)>0 \Leftrightarrow \mathbf{M}^{T}f(D)\mathbf{M}>0
\end{split}
\]
and, if Eq. \eqref{eq:KMSgap} holds true, then
$$-\int_0^{+\infty} e^{t\mathbf{Z}^T} (\mathbf{Z}^T\mathbf{M}^Tf(D) \mathbf{M}+\mathbf{M}^Tf(D) \mathbf{M}\mathbf{Z})e^{t \mathbf{Z}} dt =\mathbf{M}^{T}f(D)\mathbf{M}>0$$
as well. The converse is true in one mode: in this case the existence of a unique faithful normal invariant state implies that $D=\sigma \mathbf{1}$ for some $\sigma \in (1,+\infty)$ and $\mathbf{C}>0$ (Lemma 2 and Theorem 8 in \cite{AFP21}); from
$$\Sigma_\infty=\int_0^{+\infty}e^{t \mathbf{Z}^T}Ce^{t \mathbf{Z}}dt=\sigma \mathbf{M}^T\mathbf{M},$$
one has that $\mathbf{Z}^T\mathbf{M}^T\mathbf{M}+ \mathbf{M}^T\mathbf{M}\mathbf{Z}=\mathbf{C}/\sigma$, therefore $\mathbf{Z}^T\mathbf{M}^Tf(D) \mathbf{M}+\mathbf{M}^Tf(D) \mathbf{M}\mathbf{Z}=f(\sigma)/\sigma \mathbf{C}>0.$ We can resume what we just observed in the following lemma.
\begin{lemma}
If $d=1$ and $\TT$ admits a unique faithful invariant state, then the KMS spectral gap is always strictly positive.
\end{lemma}

It remains an interesting open problem to understand what happens in the multi-mode scenario.

\section{Long-time behavior and Environment-Induced Decoherence} \label{sec:EID}

In this section we will provide a complete description of the evolution for long times; for the sake of a simpler notation, we will study $\widetilde{\TT}$, however everything translates in terms of $\TT$ in a trivial way. Theorem \ref{thm:main} shows that the algebra of all bounded linear operators is spatially isomorphic to the product of two factors
$$B(\Gamma(\widetilde{V}_0)) \overline{\otimes} B(\Gamma(\widetilde{V}_-))$$
such that the evolution is unitary on the first factor (which does not feel any interaction with the environment), while it is dissipative and ergodic on the second factor, in the sense that
for every $x \in B(\Gamma(\widetilde{V}_-))$, one has
$${\rm w}^*\text{-}\lim_{t \rightarrow +\infty} \widetilde{\TT}_t(\mathbf{1} \otimes x)=\rho_\infty(x) \mathbf{1}.$$
Moreover, there is no interaction between the two dynamics. In this section we will show that something stronger holds true: no matter what is the initial observable, the evolution will approach for long times a unitary evolution of another observable belonging to $B(\Gamma(\widetilde{V}_0))\otimes \mathbf{1}$ and depending on the initial one. This is what is known as environment-induced decoherence (see \cite{BO03}).

\bigskip Before proving the main results, let us define the following normal quantum conditional expectation:
\begin{align*}{\cal E}:B(\FS) &\rightarrow B(\Gamma(\widetilde{V}_0)) \otimes \mathbf{1}\\
x & \mapsto \mathbb{E}_{\rho_\infty}[x]\otimes \mathbf{1},\end{align*}
where $\mathbb{E}_{\rho_\infty}$ is the unique bounded linear operator that satisfies
$$\tr(\omega \mathbb{E}_{\rho_\infty}(x))=\tr(\omega \otimes \rho_{\infty} x), \quad \forall \omega \in L^1(\Gamma(\widetilde{V}_0)), x \in B(\FS).$$
Notice that the predual of ${\cal E}$ is given by
\begin{align*}{\cal E}_*:L^1(\FS) &\rightarrow L^1(\Gamma(\widetilde{V}_0)) \otimes \rho_\infty\\
\omega & \mapsto \tr_{B(\Gamma(\widetilde{V}_-))}(\omega)\otimes \rho_\infty.\end{align*}

\begin{theo} \label{thm:eid}
Let us assume that \eqref{eq:hypo1} holds true, then one has that for every $x \in B(\FS)$:
\begin{equation} \label{eq:eid}
{\rm w}^*\text{-}\lim_{t \rightarrow +\infty} \widetilde{\TT}_t(x-{\cal E}(x))=0.
\end{equation}
Moreover, for every state $\rho$ one has

$$\lim_{t \rightarrow +\infty} \| {\widetilde{\TT}}_{t*}(\rho-{\cal E}_*(\rho))\|_1=0.$$
\end{theo}

\begin{proof}
We want to prove that Eq. \ref{eq:eid} (which we know to be true for $x$ being a Weyl operator) holds for every bounded operator. The proof follows the  same circle of ideas as in the one of quantum L\'evy continuity theorem in \cite{FP24}. Let us define the shifted total number operator
$$N:=W(w)^*\sum_{j=1}^{d}a_j^\dagger a_j W(w),$$
where $W(w)$ is the Weyl operator in the statement of Lemma \ref{lem:transmodes}; we remark that
$[\widetilde{H}_0,N]=0.$
Lemma \ref{lem:so} shows that for every initial state $\rho$, one has
\begin{equation} \label{eq:ccf}
\lim_{t \rightarrow +\infty} |\hat{\rho_t}(z)-\hat{\sigma_t}(z)|=0,\end{equation}
where $\rho_t=\widetilde{\TT}_{t*}(\rho)$ and $\sigma_t=\widetilde{\TT}_{t*}{\cal E}_*(\rho).$ Notice that for every $\alpha \geq 0$, one has that $\varphi(\alpha):=\tr(\sigma_t e^{-\alpha N})$ does not depend on $t$. This implies that the family of states $\sigma_t$ is tight and we will show that this property transfers to $\rho_t$. Indeed, one can use the following representation formula (Lemma A.2 in \cite{FP24}): for $\alpha>0$ one has
\begin{equation}\label{eq:la2}
\frac{(\theta(\alpha)+1)^{2d}}{(2\pi)^{2d}} \int_{\RR^{2d}}e^{-\theta(\alpha)(|x|^2+|y|^2)/2-2i\sigma(x+iy,w)}W(x+iy) dxdy=e^{-\alpha N}, \end{equation}
where
$$\theta(\alpha)=1+\frac{2}{e^\alpha-1}.$$
Using Eq. \eqref{eq:la2} one can see that Eq. \eqref{eq:ccf} implies that for every $\alpha \geq 0$
\begin{equation}
    \lim_{t \rightarrow +\infty} \varphi_t(\alpha)=\varphi(\alpha),
\end{equation}
where $\varphi_t(\alpha)=\tr(\rho_te^{-\alpha N}).$ We can now apply the same strategy as in the proof of Lemma A4 in \cite{FP24} to show that
\begin{itemize}
    \item for every $\epsilon>0$, there exists a finite projection $p_\epsilon$ such that for every $t \geq 0$,
$$\tr(\rho_t p_\epsilon), \tr(\sigma_t p_\epsilon)>1-\epsilon;$$
\item the following limit holds true:
$${\rm w} \text{-}\lim_{t \rightarrow +\infty } \rho_t-\sigma_t=0.$$
\end{itemize}
This two statements together imply that
$$\lim_{t \rightarrow +\infty } \|\rho_t-\sigma_t\|_1=0.$$
Indeed, for every $\epsilon>0$, there exists a $t_\epsilon>0$ such that for every $t\geq t_\epsilon$ one has
$$\|p_\epsilon(\rho_t-\sigma_t)p_\epsilon\|_1 <\epsilon.$$
Therefore,
\[\begin{split}|\tr((\rho_t-\sigma_t)x)| &\leq \tr(((\rho_t-\sigma_t)-p_\epsilon (\rho_t-\sigma_t)p_\epsilon)x) \quad (I)\\
&+\tr((p_\epsilon (\rho_t-\sigma_t)p_\epsilon)x). \quad (II)
\end{split}\]
Notice that $(I) \leq 6\epsilon \|x\|_\infty$ and $(II) \leq \epsilon \|x\|_\infty$ and we are done.
\end{proof}

As we already hinted, Theorem \ref{thm:eid} has a deep physical interpretation, in that it shows that when $\TT$ admits a normal invariant state, environment-induced decoherence (EID) takes place; let us first recall the definition of EID following \cite{BO03}. The dynamics $\TT$ shows EID if there exists a decomposition
$$B({\frak h})={\cal M}_1\oplus {\cal M}_2$$
such that
\begin{itemize}
\item ${\cal M}_1$ is a $\TT$-invariant $W^*$-algebra,
\item ${\cal M}_2$ is a $\TT$-invariant and $^*$-invariant ${\rm w}^*$-closed subspace,
\item $\TT$ acts on ${\cal M}_1$ as a family of $^*$-automorphisms,
\item for every $x \in {\cal M}_2$,
$${\rm w}^*\text{-}\lim_{t \rightarrow +\infty} \TT_t(x)=0$$
\end{itemize}
${\cal M}_1$ is known as the algebra of effective observables, while ${\cal M}_2$ is the set of those observables which after a long time cannot be detected anymore. The following result is just a reformulation of Theorem \ref{thm:eid}.

\begin{coro} \label{coro:EID}
    If \eqref{eq:hypo1} is true, then $\widetilde{\TT}$ undergoes EID with
    $${\cal M}_1={\cal N}(\widetilde{\TT})=B(\Gamma(\widetilde{V}_0)\otimes \mathbb{1} \text{ and } {\cal M}_2=({\rm Id}-{\cal E})(B(\Gamma(\CC^d))).$$
\end{coro}

To the best of our knowledge, the following are the cases in which EID is known to take place:
\begin{itemize}
    \item when the semigroup is uniformly continuous and admits a faithful normal invariant state (Theorem 12 and 22 in \cite{SU24}),
    \item when the semigroup admits a faithful normal invariant state (the authors of \cite{CJ20} prove it in the case of discrete time dynamics, but the proof should carry to the continuous time setting),
    \item when the semigroup acts on a finite dimensional matrix algebra and the linear space generated by eigenvectors of the generator corresponding to purely imaginary eigenvalues coincides with $\NN$ (Theorem 6 in \cite{CSU13}).
\end{itemize}
Corollary \ref{coro:EID} does not follows from any of the previously mentioned general results, since it holds also for semigroups without any faithful normal invariant state.

\section{Decoherence Speed} \label{sec:ds}

In the previous section, we showed that for every $x \in B(\hh)$, one has
$${\rm w}^*\text{-}\lim_{t \rightarrow +\infty}\widetilde{\TT}_t(x-{\cal E}(x))=0.$$
In this section, we will provide some results concerning the speed at which this decay happens \textbf{under the hypothesis that $\TT$ admits a faithful invariant state.} The convergence will be considered in a family of noncommutative $L^2$ spaces induced by any faithful normal invariant state for $\widetilde{\TT}$. The main result of this section is that, as one may expect, the slowest rate of convergence can be determined by looking only at the action of the semigroup on $\mathbb{1}\otimes B(\Gamma(\widetilde{V}_-))$, on which it acts in an ergodic way.

We will make use in several places of the equivalence between Poincar\'e-type inequalities and uniform exponential decay in $L^2$; while it is fairly easy to find references of this fact for selfadjoint ergodic semigroups (see, for instance, Chapter 2 in \cite{GZ03}), we could only find one reference treating ergodic semigroups which are not selfadjoint (see Theorem 2.3 \cite{Li89}) and none for the theorem in the generality that we need, i.e. in the case of a semigroup which is neither ergodic, nor selfadjoint. Theorem \ref{theo:pi} in the Appendix serves this purpose.

\bigskip Let us introduce the framework in which we are going to establish the decoherence speed. Given any faithful normal invariant state $\omega_\infty:=\omega \otimes \rho_\infty$ for $\widetilde{\TT}$ and $s \in [0,1]$, we can define the following scalar product on $B(\FS)$:
\begin{equation} \label{eq:inner}
\langle x, y \rangle_s:=\tr(\omega_\infty^s x^* \omega_\infty^{1-s}y), \quad x,y \in B(\FS).
\end{equation}
We define $\LD$ as the Hilbert space given by the completion of $B(\FS)$ with respect to such inner product. One can see that $B(\FS)$ embeds continuosly into $\LD$ and that $\widetilde{\TT}$ extends to a strongly continuous contraction semigroup $\widetilde{T}$ on $\LD$ (see Theorem 2.3 in \cite{CF00}).

Notice that $B(\FS)=B(\Gamma(\widetilde{V}_0)) \overline{\otimes}B(\Gamma(\widetilde{V}_-))$ and for all $x,a \in B(\Gamma(\widetilde{V}_0))$ and $ y,b \in B(\Gamma(\widetilde{V}_-))$, one has
$$\langle x \otimes y, a \otimes b \rangle_s=\tr(\omega^s x^* \omega^{1-s}a)\tr(\rho_\infty^s y^* \rho_\infty^{1-s}b)=:\langle x ,a\rangle_{1,s}\langle y, b \rangle_{2,s}.$$
If we define $L^2_s(\omega)$ and $L^2_s(\rho_\infty)$ in the same way as we did for $L^2_s(\omega_\infty)$, one has that
$$\LD=L^2_s(\omega)\otimes L^2_s(\rho_\infty).$$
Let us define
$$\widetilde{\TT}_t^1(x):=e^{-it\widetilde{H}_0} x e^{it\widetilde{H}_0}, \quad x \in B(\Gamma(\widetilde{V}_0))$$
and $\widetilde{\TT}^2$ as the GQMS acting on $B(\Gamma(\widetilde{V}_-))$ corresponding to $\widetilde{H}_-$ and $\widetilde{L}_l$ ($\widetilde{H}_0$, $\widetilde{H}_-$ are defined in Theorem \ref{thm:main}). $\widetilde{\TT}^1$ (resp. $\widetilde{\TT}^2$) extends to a strongly continuous contraction semigroup $\widetilde{T}^1$ (resp. $\widetilde{T}^2$) on $L^2_s(\omega)$ (resp. $L^2_s(\rho_\infty)$; we denote by $\widetilde{L}_1$ and $\widetilde{L}_2$ the generators of $\widetilde{T}^1$ and $\widetilde{T}^2$, respectively. One has that $\widetilde{\TT}_t=\widetilde{\TT}_t^1 \otimes \widetilde{\TT}_t^2$ and, consequently, $\widetilde{T}_t=\widetilde{T}^1_t \otimes \widetilde{T}^2_t$ for every $t \geq 0$.

$\widetilde{T}^2$ is ergodic, in the sense that for every $x \in L^2_s(\rho_\infty)$ such that $\langle x, \mathbf{1} \rangle_{2,s}=0$, one has
$$\lim_{t \rightarrow +\infty} \|\widetilde{T}^2_t(x)\|_{2,s}=0.$$

In this case, Theorem \ref{theo:pi} shows that the optimal constant $\lambda\geq 0$ such that
$$\|\widetilde{T}^2_t(x)\|_{2,s} \leq e^{-\lambda t}\|x\|_{2,s}, \quad \forall x \in D( L^2_s(\rho_\infty), \, \langle x, \mathbf{1} \rangle_{2,s}=0$$
is given by
\begin{equation} \label{eq:gapL2}
{\rm gap}_s(\widetilde{L}_2):=\inf\left \{\mathscr{E}_{2,s}(x,x):x \in D(\widetilde{L}_2)\subset L^2_s(\rho_\infty),\,\|x\|=1,\,\langle x, \mathbf{1} \rangle_{2,s}=0\right \},\end{equation}
where $\mathscr{E}_{2,s}$ is the positive quadratic form defined as
\begin{align*}
&\left .\mathscr{E}_{2,s}(x)=-\frac{d}{dt}\frac{\|\widetilde{T}^2_t(x)\|^2_{2,s}}{2}\right |_{t=0}=-\Re(\langle \widetilde{L}_2(x),x \rangle_{2,s}).
\end{align*}
Using Proposition 2.2 in \cite{CF00} one can easily check that ${\cal E}$ also extends to the orthogonal projection $E$ on the closure of $\NN$ in  $\LD$. The following result holds true.

\begin{theo} The optimal constant $\lambda \geq 0$ such that
\begin{equation} \label{eq:gapL}
\|\widetilde{T}_t(x)\|_{s} \leq e^{-\lambda t}\|x\|_{s}, \quad \forall x \in L^2_s(\omega_\infty), \,\, E(x)=0
\end{equation}
is given by ${\rm gap}_s(L_2)$ defined in Eq. \eqref{eq:gapL2}.
\end{theo}
Notice that the optimal decay constant cannot be bigger than ${\rm gap}(L_2)$, since it is the optimal decay constant of $\widetilde{T}^2$ and $\widetilde{T}_t(\mathbf{1} \otimes x)=\mathbf{1} \otimes \widetilde{T}^2_t(x)$ for every $t \geq 0$, therefore the nontrivial statement of the theorem is that the optimal constant is not strictly smaller than ${\rm gap}_s(L_2)$.

A physical interpretation of this theorem is that the decay rate is completely determined by the dissipative part (notice that it does not even depend on the choice of the reference invariant state); determining ${\rm gap}_s(L_2)$, then, reduces to the problem of finding the spectral gap of an ergodic GQMS, which has been addressed in \cite{FPSU24}. To be precise, the authors analysed the case of KMS ($s=1/2$) and GNS ($s=1$) inner products, but it should be possible to deal with the remaining cases by standard results in complex interpolation theory (\cite{Ko84}).

\begin{proof}

Let us consider $D=L^2_{1,s}(\omega) \otimes D(\widetilde{L}_2)$; for every $x=\sum_{i=1}^{k}a_i \otimes b_i \in D$,
$$\mathscr{E}_{s}(x):=\left .-\frac{d}{dt}\frac{\|\widetilde{T}_t(x)\|^2_s}{2}\right |_{t=0}$$
is well defined and
\begin{equation} \label{eq:specquad}
\mathscr{E}_{s}\left (\sum_{i=1}^{k}a_i \otimes b_i\right )=\sum_{i,j=1}^{k} \langle a_i,a_j\rangle_{1,s} \mathscr{E}_{2,s}(b_i,b_j).
\end{equation}
Indeed, notice that
\[\begin{split}
    &\left \| \widetilde{T}_t \left ( \sum_{i=1}^{k}a_i \otimes b_i\right ) \right \|_{s}^2=\left \|  \sum_{i=1}^{k}\widetilde{T}^1_t(a_i) \otimes \widetilde{T}_t^2(b_i) \right \|_{s}^2=\\
    &\sum_{i,j=1}^{k} \langle \widetilde{T}_t^1(a_i),\widetilde{T}_t^1(a_j) \rangle_{1,s}\langle \widetilde{T}_t^2(b_i),\widetilde{T}_t^2(b_j) \rangle_{2,s}. 
\end{split}\]
However, since $\omega$ commutes with $e^{itH_0}$, one has that $\widetilde{\TT}^1$ is a semigroup of unitary operators. Indeed, for every $x,y \in B(\Gamma(V_0))$ the following holds true:
\[\begin{split}
    &\langle \widetilde{T}_t^1(x),\widetilde{T}_t^1(y) \rangle_{1,s}=\tr(\omega^s e^{-it\widetilde{H}_0}x^* e^{it\widetilde{H}_0} \omega^{1-s}e^{-it\widetilde{H}_0}y e^{it\widetilde{H}_0})=\\
    &\tr(\omega^s e^{-it\widetilde{H}_0}x^* e^{it\widetilde{H}_0} \omega^{1-s}e^{-it\widetilde{H}_0}y e^{it\widetilde{H}_0})=\tr(\omega^s x^*  \omega^{1-s}y )=\langle x, y \rangle_{1,s}
\end{split}\]
and this extends to the whole $L^2_s(\omega)$ by a density argument. Therefore,
$$\left \| \widetilde{T}_t \left ( \sum_{i=1}^{k}a_i \otimes b_i\right ) \right \|_{s}^2=\sum_{i,j=1}^{k} \langle a_i,a_j \rangle_{1,s}\langle \widetilde{T}_t^2(b_i),\widetilde{T}_t^2(b_j) \rangle_{2,s}$$
and we get Eq. \eqref{eq:specquad}.

Since $\widetilde{T}_t=\widetilde{T}_t^1\otimes \widetilde{T}_t^2$ and $D(\widetilde{L}_2)$ is $T_t^2$-invariant, we have that $\widetilde{T}_t(D) \subseteq D$ for every $t \geq 0$. Moreover, $\sum_{i=1}^k a_i \otimes b_i \in D$, one can check that
$$({\rm Id}-E)\left (\sum_{i=1}^k a_i \otimes b_i \right )=\sum_{i=1}^k a_i \otimes (b_i-\rho_\infty(b_i)) \in D.$$
Therefore $\overline{D}:=({\rm Id}-E)(D) \subset D$ and it is spanned by elements of the form $a \otimes b \in D$ such that $\rho_\infty(b)=0$. Notice that $D$ is dense in $\LD$, therefore $\overline{D}$ is dense in $({\rm Id}-E)(\LD)$. Theorem \ref{theo:pi} shows that the optimal constant for Eq. \eqref{eq:gapL} to hold is given by
$${\rm gap}(\widetilde{L}):=\inf\{\mathscr{E}_s(f): \, f \in D, \|f\|=1, \, E(f)=0\}.
$$

In order to conclude, we need to prove that ${\rm gap}(\widetilde{L}) \geq {\rm gap}(\widetilde{L}_2)$. Given $\sum_{i=1}^k a_i \otimes b_i \in \overline{D}$, let us define 
$$A_{i,j}=\langle a_i,a_j \rangle_{1,s}, \quad B_{i,j}=\langle b_i,b_j \rangle_{2,s}, \quad C_{i,j}=\mathscr{E}_{2,s}(b_i,b_j).$$
Notice that $A,B,C \geq 0$, moreover the definition of ${\rm gap}(\widetilde{L}_2)$ implies that $C \geq {\rm gap}(\widetilde{L}_2) B$. Therefore, if we use $*$ to denote the Hadamard product, one has that $A*C \geq {\rm gap}(\widetilde{L}_2) A *B$ as well (this is due to bilinearity of Hadamard product and to Schur product theorem). Notice that
$$\mathscr{E}_s \left (\sum_{i=1}^k a_i \otimes b_i \right )=\langle \underline{1}, A*C \underline{1} \rangle \geq {\rm gap}(\widetilde{L}_2) \langle \underline{1}, A*B \underline{1} \rangle={\rm gap}(\widetilde{L}_2) \left \|\sum_{i=1}^{k}a_i \otimes b_i \right \|_{s}^2$$
and we are done.
\end{proof}

\section{Ergodic Theory, Transience and Recurrence} \label{sec:et}

In this section we will study the limit of ergodic means and we will show that an ergodic theorem holds true. Moreover, we will be able to completely determine the decomposition of the system Hilbert space related to recurrence and transience (see Theorem 9 in \cite{Um06} and Remark 12 in \cite{CG21}). 

\bigskip Given an initial state $\rho$, one can easily see from Theorem \ref{thm:eid} that $\widetilde{\TT}_{t*}(\rho)$ might not converge to any limit. On the other end, averaging in time usually ensures a better behavior, as the following theorem shows.

Let us first define the following conditional expectation:
\begin{align*}{\cal F}_*:L^1(\FS) &\rightarrow L^1(\FS)\\
\omega & \mapsto \sum_{x \in {\rm Sp}(\widetilde{H}_0)} P_x\tr_{B(\Gamma(\widetilde{V}_-))}(\omega)P_x\otimes \rho_\infty,\end{align*}
where $P_x$ denotes the spectral projection of $\widetilde{H}_0$ corresponding to the eigenvalue $x$.

\begin{theo} \label{thm:ergo}
For every initial state $\rho$, one has
$$\lim_{t \rightarrow +\infty} \frac{1}{t}\int_0^t\widetilde{\TT}_{s*}(\rho)ds={\cal F}_*(\rho)$$
in trace norm.
\end{theo}
\begin{proof}
By Theorem \ref{thm:eid}, one has that for every initial state $\rho$,
\[
\lim_{t \rightarrow +\infty}\|\widetilde{\TT}_t(\rho)-e^{it\widetilde{H}_0}\rho_1 e^{-it\widetilde{H}_0} \otimes \rho_{\infty}\|_1=0,
\]
where $\rho_1=\tr_{B(\Gamma(\widetilde{V}_-))}(\rho)$. Therefore, it suffices to show that
\[
\lim_{t \rightarrow +\infty}\frac{1}{t} \int_0^t e^{is\widetilde{H}_0}\rho_1 e^{-is\widetilde{H}_0}ds =\rho_{\widetilde{H}_0}, \quad \rho_{\widetilde{H}_0}:=\sum_{x \in {\rm Sp}(\widetilde{H}_0)}P_x \rho_1 P_x.
\]
${\rm Sp}(\widetilde{H}_0)$ is a countable set, hence we can list its elements $\{x_1,x_2,\dots\}$; let $Q_n$ the spectral projection corresponding to $\{x_1,\dots, x_n\}$. Since $Q_n \uparrow \mathbf{1}$ in the strong operator topology, one has that for every $\epsilon>0$, there exists $N_\epsilon \in \mathbb{N}$ such that for every $n \geq N_\epsilon$ and for every $t \geq 0$
$$\tr(e^{it\widetilde{H}_0}\rho_1 e^{-it\widetilde{H}_0} Q_n)=\tr(\rho_1 Q_n)=\tr(\rho_{\widetilde{H}_0}Q_n) \geq 1-\epsilon.$$

One has that for every $t >0$
$$\left \|\frac{1}{t} \int_0^t e^{is\widetilde{H}_0}\rho_1 e^{-is\widetilde{H}_0}ds -\rho_{\widetilde{H}_0}\right \|_1 \leq \left \|\frac{1}{t} \int_0^t e^{is\widetilde{H}_0}\rho^\epsilon_1 e^{-is\widetilde{H}_0}ds -\rho^\epsilon_{\widetilde{H}_0}\right \|_1+6\epsilon,$$
where
$$\rho^\epsilon_1=Q_{N_\epsilon}\rho_1Q_{N_\epsilon}, \quad \rho^\epsilon_{\widetilde{H}_0}=Q_{N_\epsilon}\rho_{\widetilde{H}_0}Q_{N_\epsilon}.$$
Let
$$0 <\Delta_\epsilon:=\inf\{x-y: x, y \in \{x_1,\dots, x_{N_\epsilon}\}, \, x \neq y\},$$
then
$$\left \|\frac{1}{t} \int_0^t e^{is\widetilde{H}_0}\rho^\epsilon_1 e^{-is\widetilde{H}_0}ds -\rho^\epsilon_{\widetilde{H}_0}\right \|_1 \leq \sum_{\substack{x,y \in \{x_1,\dots, x_{N_\epsilon}\}\\ x \neq y}}|\tr(P_x \rho_1 P_y)| \frac{2}{\Delta_\epsilon t}$$
can be made arbitrarily small picking $t$ big enough and we are done.
\end{proof}

The previous results, allows to completely determine the splitting of the system Hilbert space introduced in \cite{Um06} into positive recurrent, null recurrent and transient spaces. We recall that such decomposition of the Hilbert space generalizes the decomposition of the states of a classical Markov process into positive recurrent, null recurrent and transient states (see, for instance, Chapter 3 in \cite{Norris}); even in the setting of quantum Markov evolutions, it plays a fundamental role in the study of the qualitative behavior of the semigroup and is deeply related to the notions of accessibility, normal invariant states, conserved quantities and strong symmetries.

Let us start from the positive recurrent subspace, which already appeared in the literature in the $70$s (\cite{EHK78}) and it is defined as
$${\frak R}_+:=\sup\{{\rm supp}(\rho): \, \rho \text{ is a normal invariant state}\}.$$

From Theorem \ref{theo:invst}, one immediately sees that the positive recurrent space is given by
$${\frak R}_+:=\Gamma(\widetilde{V}_0)\otimes {\rm supp}(\rho_\infty).$$

The transient space can be defined as the supremum of those projections in which the system spends a small amount of time no matter what is the initial state (Proposition 8 in \cite{Um06}); in formulas,
$${\frak T}:=\sup\left \{ p \text{ projection }: \exists \, C>0, \, \forall \text{ state }\rho, \, \int_{0}^{+\infty} \tr(\widetilde{\TT}_{s*}(\rho)p)ds \leq C\right \}.$$

Theorem \ref{thm:ergo} implies the following.

\begin{coro}
    ${\frak R}_+^\perp$ is the transient space ${\frak T}$.
\end{coro}
\begin{proof}
Since ${\cal F}_*(\rho)$ is a state for every initial state $\rho$, Theorem 2.3.23 in \cite{Gi22} implies that the absorption operator (\cite[Definition 2]{CG21}) corresponding to ${\frak R}_+$, which we denote by $A({\frak R}_+)$, is the identity operator. Let $p_{{\frak R}_+}$ be the orthogonal projection onto ${\frak R}_+$; Theorem 14 in \cite{CG21} shows that ${\rm supp}(A({\frak R}_+)-p_{{\frak R}_+})$ is contained in the transient space, however notice that
$$A({\frak R}_+)-p_{{\frak R}_+}=\mathbf{1}-p_{{\frak R}_+},$$
therefore ${\rm supp}(A({\frak R}_+)-p_{{\frak R}_+})={\frak R}_+^\perp$ is the transient space.
\end{proof}

Since the previous corollary implies that the null recurrent space is trivial, we will not provide any information about it, referring the interested reader to \cite{FR03,Um06,CG21,Gi22,Gi22a}.

\section{Classical Ornstein-Uhlenbeck semigroups} \label{sec:cOU}

In this section, after recalling the definition of classical Ornstein-Uhlenbeck semigroups, we will comment on the problem of irreducibility and the characterization of those semigroups admitting an invariant measure in this setting, comparing the results with those in the quantum case. Since the focus of this work is on finitely many modes, we will only consider Ornstein-Uhlenbeck semigroups in finite dimensions (\cite{LMP20}).

\bigskip Let us consider the following linear stochastic differential equation:
\begin{equation} \label{eq:SDE}
dX_t=(AX_t+b)dt+BdW_t,
\end{equation}
where $A \in M_d(\RR)$, $B \in M_m(\RR)$, $b \in \RR^d $ and $W_t$ is a $m$-dimensional Brownian motion. For every initial condition $x$, there exists a unique strong solution $X_t^x$ that is a Markov process; the associated semigroup $T:=(T_t)_{t \geq 0}$ is the classical Ornstein-Uhlenbeck semigroup which acts on functions of the form $x \mapsto e^{i\langle z, x \rangle }$ in the following way:

\begin{equation} \label{eq:cOU}
T_t(e^{i\langle z, \cdot \rangle})(x)=\exp\left (i\int_0^t \langle z,e^{sA}b \rangle ds-\frac{1}{2}\int_0^t \langle z, e^{sA}BB^Te^{sA^T} z \rangle ds\right )e^{i \langle z, e^{tA}x \rangle}, \quad \forall z \in \RR^d.
\end{equation}

Ornstein-Uhlenbeck semigroups represent the class of all evolutions preserving the set of Gaussian measures. The analogy between Eq. \eqref{eq:aW} and Eq. \eqref{eq:cOU} is evident; notice that $A^T$ corresponds to $\mathbf{Z}$ and $BB^T$ to $\mathbf{C}$.

In the case of $b=0$, a characterization of those semigroups admitting an invariant measure and a description of the set of invariant measures in the classical case can be found in Theorems 4.1 and 4.2 of \cite{ZS70}; below we report the facts that are relevant to this work. Before doing this, we need to introduce some more notation: let us define
$$S_-:=\left \{ x \in \mathbb{R}^d: \lim_{t \rightarrow +\infty}e^{tA}x=0\right \}.$$

\begin{theo} \label{thm:cim1}
  The following conditions are equivalent:
  \begin{enumerate}
  \item $T$ admits an invariant measure,
  \item the linear span of the columns of $B$, $AB, \dots A^{d-1}B$ is contained in $S_-$,
  \item $\Sigma_{\infty}:=\lim_{t \rightarrow +\infty}\int_0^t e^{sA}BB^Te^{sA^T}ds$ exists (finite).
  \end{enumerate}
  In this case, every invariant measure is given by the convolution $\nu * {\cal N}(0,\Sigma_\infty)$ where $\nu$ is any invariant measure for Eq. \eqref{eq:SDE} with $B=0$.
\end{theo} 
Let $\mathscr{L}$ denote the $d$-dimensional Lebesgue measure; if we specialize Theorem \ref{thm:cim1} to the case when $T_t$ admits an invariant measure which is absolutely continuous with respect to $\mathscr{L}$, using the same techniques as in this work, one obtains the counterpart of Theorems \ref{thm:main} and \ref{theo:invst}. Before stating these results, we need to introduce the equivalence class of an Ornstein-Uhlenbeck semigroup: for every linear transformation $M:\mathbb{R}^d \rightarrow \mathbb{R}^d$, we can consider the stochastic process $MX_t$ and the corresponding semigroup $T^M$. The equivalence class of an Orstein-Uhlenbeck semigroup $T$ is defined as
$$[T]:=\{T^M: \, M :\RR^d \rightarrow \RR^d  \text{ linear transformation}\}.$$
\begin{theo}
   The following conditions are equivalent:
   \begin{enumerate}
   \item $T_t$ admits an invariant measure which is absolutely continuous with respect to $\mathscr{L}$,
   \item there exist an integer $1 \leq d_0 \leq d/2$, non zero real numbers $\phi_1,\dots, \phi_{d_0}$ and a representative $\widetilde{T} \in [T]$ such that
$$\widetilde{A}=\begin{pmatrix} \widetilde{A}_0 & O\\
      O  & \widetilde{A}_- \end{pmatrix}, \quad \widetilde{B}\widetilde{B}^T=\begin{pmatrix} O & O\\
     O & (\widetilde{B}\widetilde{B}^T)_- \end{pmatrix}, \quad \widetilde{b}=\begin{pmatrix} \widetilde{b}_0 \\ \widetilde{b}_- \end{pmatrix} $$
      where
      \begin{itemize}
          \item
          $$\widetilde{A}_0=\begin{pmatrix}O & O &O \\
          O & O & -\Phi \\
      O& \Phi & O \end{pmatrix}$$ and
      $\Phi$ is the diagonal matrix with diagonal entries $\phi_1, \dots, \phi_{d_0}$,
          \item ${\rm Sp}(\widetilde{A}_-) \subseteq \{ \Re(z)<0\} $,
          \item $\widetilde{b}_0 \in {\rm supp}(\widetilde{A}_0)$.
      \end{itemize}
      
   \end{enumerate}

Let $k:=\dim(\ker(\widetilde{A}))$ and let us write $x =(x_1,x_2) \in \RR^d$ where $x_1$ contains the first $k+2d_0$ coordinates. In this case, every invariant measure for $\widetilde{T}$ which is absolutely continuous with respect to $\mathscr{L}$ has a density of the form $f(x)=f_1(x_1)f(x_2)$ where $f_2(x)$ is a Gaussian density with mean and covariance given by
$$m_\infty:=\int_0^{+\infty}e^{t\widetilde{A}_-}\tilde{b}_-dt, \quad \Sigma_\infty:=\int_0^{+\infty}e^{t\widetilde{A}_-}(\tilde{B}\tilde{B}^T)_-e^{t\widetilde{A}^T_-}dt$$
and $f_1$ is invariant under the rotations generated by $\widetilde{A}_0$, i.e.
$$f_1(e^{t\widetilde{A}^T_0}x_1)=x_1, \quad \forall t \geq 0.$$
\end{theo}

We point out that the proof of the result in the classical case is easier due to commutativity or, equivalently, the lack of the symplectic structure (which required the extra work in Lemmas \ref{lem:so} and \ref{lem:scb}).

Finally, let us comment on the irreducibility condition in Corollary \ref{coro:irr}: using Proposition 2 in \cite{Za85} one can show that the process is $\mathscr{L}$-irreducible (see Section 4.2 in \cite{MTG09}) if and only if
\begin{equation} \label{eq:clirrcond}
\text{there are not nontrivial $A^T$-invariant subspaces in }\ker(BB^T).
\end{equation}
This is the classical counterpart of condition in Eq. \eqref{eq:irrcond} with the only difference that there is no contribution from the symplectic matrix.


\section*{Aknowledgements.}
The authors are grateful to Franco Fagnola for several useful discussions and part of their work was supported by the University of Nottingham and the University of Tübingen’s funding as part of the Excellence Strategy of the German Federal and State Governments, in close collaboration with the University of Nottingham. The authors are also grateful to the organizers of the Birs Workshop on Quantum Markov Semigroups and Channels held in Oaxaca for the opportunity to interact among them and with other experts in the field. F.G. has been partially supported by the MUR grant Dipartimento di Eccellenza
2023–2027 of Dipartimento di Matematica, Politecnico di Milano and by "INdAM -GNAMPA Project" CUP E53C23001670001. D.P. has been partially supported by the MUR grant Dipartimento di Eccellenza of Dipartimento di Matematica, Università degli studi di Genova and is a member of the INdAM-GNAMPA group.

\appendix

\section{Equivalence between Poincar\'e inequality and exponential decay of the norm in the general case} \label{app:PI}
Let ${\cal H}$ denote an Hilbert space and $S_t$ a strongly continuous contraction semigroup on ${\cal H}$, with generator $A$. Let us assume that there exists a non trivial orthogonal projection $E$ such that $S_t E=E S_t$.

Let us define
\begin{equation}\label{eq:gap1}
{\rm gap}(A):=\max\{\alpha \geq 0 : \|S_tf\| \leq e^{-\alpha t}, \, \forall f \in {\cal H}, \, \|f\|=1, \, E(f)=0\}\end{equation}

\begin{theo} \label{theo:pi}
For every linear space $D \subseteq {\cal H}$ such that
\begin{itemize}
    \item $\mathscr{E}(f):=\left.-\frac{d}{dt} \frac{\|S_tf\|^2}{2} \right|_{t=0}$ is well defined for every $f \in D$,
    \item $S_t(D) \subseteq D$ for all $t \geq 0$,
    \item $({\rm Id}-E)(D)\cap D$ is dense in the range of ${\rm Id}-E$,
\end{itemize}
one has
\begin{equation} \label{eq:gap2}
  {\rm gap}(A)=\inf\{\mathscr{E}(f) : \forall f \in D, \, \|f\|=1, \, E(f)=0\}.  
\end{equation}
In particular,
\begin{equation}\label{eq:gap3}
{\rm gap}(A)=\inf\{-\Re(\langle f, A f \rangle : \forall f \in D(A), \, \|f\|=1, \, E(f)=0\}.\end{equation}
\end{theo}
\begin{proof}
Let $\lambda$ denote the RHS of Eq. \eqref{eq:gap2}. Notice that, due to the contractivity of $S_t$, one has that $\lambda \geq 0$.

\bigskip First we prove that ${\rm gap}(A) \leq \lambda$. Indeed, for every $f\in D$, $\|f\|=1$, $E(f)=0$, one has
$$\mathscr{E}(f)=\lim_{t \rightarrow 0^+} \frac{1-\|S_tf\|^2}{2t} \geq \lim_{t \rightarrow 0^+} \frac{1-e^{-2{\rm gap}(A)t}}{2t} ={\rm gap}(A).$$
Therefore, ${\rm gap}(A) \leq \lambda$.

\bigskip Now, we are going to prove the reverse inequality. Let us consider $f \in D$ such that $\|f\|=1$ and $E(f)=0$; notice that by hypotheses $S_t f \in D$ and $ES_t(f)=0$ for every $t \geq 0$. By the definition of $\lambda$, one has
$$2\lambda \|S_tf\|^2\leq \mathscr{E}(S_tf)= -\frac{d}{dt}\|S_t f \|^2.$$
Therefore,
$$\frac{d}{dt}(\|S_t f \|^2e^{2\lambda t})=2(\lambda \|S_t f \|^2-\mathscr{E}(S_tf)) e^{2\lambda t} \leq 0,$$
which means that $\|S_t f \|^2e^{2\lambda t}$ is a monotone non-increasing function, hence
$$\|S_t f \|^2e^{2\lambda t} \leq 1 \Leftrightarrow \|S_t f \| \leq e^{-\lambda t}.$$
By the density of $({\rm Id}-E)(D)\cap D$ in the range of ${\rm Id-E}$, it follows that $\lambda \leq {\rm gap}(A)$ as well and we are done.

The second characterization follows from the fact that $D(A)$ satisfies all the requirements for $D$ above: indeed, for all $f \in D(A)$, $\|S_tf\|^2$ is differentiable and
$$\mathscr{E}(f)=-\Re(\langle f,Af \rangle).$$
Moreover it follows from the general theory of strongly continuous semigroups that $S_t(D(A)) \subseteq D(A)$ for all $t \geq 0$  and that, since $S_tE =ES_t$ for every $t \geq 0$, $({\rm Id}-E)(D(A))\cap D(A)$ is dense in the range of ${\rm Id}-E$.
\end{proof}

\bibliographystyle{abbrv}
\bibliography{biblio.bib}
\end{document}